\documentclass[11pt]{article}
\usepackage{amsmath,amsthm,amsfonts,amssymb,mathrsfs,epsfig,bm}
\usepackage[usenames]{color}
\usepackage[colorlinks=true]{hyperref}
\usepackage{graphicx}
\usepackage[scale=.8]{geometry}
\usepackage{ulem}
\usepackage{textcomp}
\usepackage{amsmath}
\usepackage{esint}

\newcommand{\sign}[1]{\mathrm{sgn}(#1)}

\parskip        \medskipamount

\newtheorem{theorem}{Theorem}

\newtheorem{lemma}{Lemma}

\newtheorem{corollary}{Corollary}
\newtheorem{proposition}{Proposition}
\newtheorem{definition}{Definition}

\theoremstyle{remark}
\newtheorem{remark}[theorem]{Remark}

\newtheorem{example}{Example}

 \def\beqlb{\begin{eqnarray}}\def\eeqlb{\end{eqnarray}}
 \def\beqnn{\begin{eqnarray*}}\def\eeqnn{\end{eqnarray*}}

 \def\ar{\!\!\!&}

 \def\mcr{\mathscr}

 \def\proof{\noindent{\it Proof.~~}}\def\qed{\hfill$\Box$\medskip}

\def\R{\mathbb{R}}

\def\P{\mathbb{P}}
\def\E{\mathbb{E}}

\def\FF{\mathscr{F}}

\renewcommand{\Phi}{\varPhi}
\renewcommand{\epsilon}{\varepsilon}

\renewcommand{\liminf}{\varliminf}
\renewcommand{\d}{\text{\rm\,d}}

\definecolor{mygray}{gray}{0.9}
\definecolor{deeppink}{RGB}{255,20,147}
\definecolor{mygreen}{rgb}{0.05, 0.576, 0.03}
\definecolor{myred}{rgb}{0.768, 0.09, 0.09}

\long\def\symbolfootnote[#1]#2{\begingroup
\def\thefootnote{\fnsymbol{footnote}}\footnote[#1]{#2}\endgroup}

\usepackage{MnSymbol}

\usepackage{authblk}
\newcommand{\ddr}{\mathrm{d}}

\begin{document}

\title{\bf Limit theorems for continuous-state branching processes with immigration}
\author{{\textbf{Cl\'ement Foucart}\thanks{LAGA Universit\'e Sorbonne Paris Nord E--mail: foucart@math.univ-paris13.fr}\ \ \&\ \textbf{Chunhua Ma} }\thanks{School of Mathematical Sciences and LPMC, Nankai University E--mail: mach@nankai.edu.cn}\ \ \&\
\textbf{Linglong Yuan}\thanks{Department of Mathematical Sciences, University of Liverpool; Department of Mathematical Sciences, Xi'an Jiaotong-Liverpool University; E--mail:  yuanlinglongcn@gmail.com   }}
\maketitle
\begin{abstract}
A continuous-state branching process with immigration, whose branching mechanism is $\Psi$ and immigration mechanism $\Phi$, CBI$(\Psi,\Phi)$ for short, may have two different asymptotic regimes  depending on whether $\int_{0}\frac{\Phi(u)}{|\Psi(u)|}\ddr u<\infty$ or $\int_{0}\frac{\Phi(u)}{|\Psi(u)|}\ddr u=\infty$.  When $\int_{0}\frac{\Phi(u)}{|\Psi(u)|}\ddr u<\infty$, CBIs either have a limit distribution or a growth rate dictated by the branching dynamics. When $\int_{0}\frac{\Phi(u)}{|\Psi(u)|}\ddr u=\infty$, immigration  overwhelms branching dynamics. Asymptotics in this case are studied  via a non-linear time-dependent renormalization in law. Three regimes of weak convergence are  exhibited. 
Processes with critical branching mechanisms subject to a regular variation assumption are studied.  
This article proves and extends results stated by M. Pinsky in \textit{Limit theorems for continuous state branching processes with immigration} [Bull. Amer. Math. Soc. \textbf{78} (1972)]. 
\end{abstract}
 \vspace{8pt} \noindent {\bf Key words.}
{Continuous-state branching processes}, {Immigration}, {Grey's martingale}, {Limit distribution}, {Non-linear renormalization}, {Regularly varying functions}.


\section{Introduction}\label{CBI}
Continuous-state branching processes with immigration (CBI) have been defined by Kawazu and Watanabe \cite{KAW}. They are scaling limits of Galton-Watson Markov chains with immigration, see e.g. \cite[Theorem 2.2]{KAW}. Recent years have seen renewed interest in this class of Markov processes. They appear for instance as strong solutions of some stochastic differential equations with jumps, see Dawson and Li \cite{DawsonLi}, and from a more applied point of view, form an important subclass of the so-called affine processes, which are known in the financial mathematics setting for modelling interest rates, see \cite{duffie2003}. We mention for instance the works in this direction of Jiao et al. \cite{Jiao2017} and Barczy et al. \cite{Barczy} where certain CBI processes are studied from a statistical point of view. 

The asymptotic behaviors of Galton-Watson processes with immigration have been extensively studied since the seventies. We refer to the works of Cohn \cite{Cohn}, Heathcote \cite{Heathcote}, Heyde \cite{heyde1970}, Pakes \cite{Pakes} and Seneta \cite{MR0270460} and to their references. Transience and recurrence of CBIs have been characterized by Duhalde et al. \cite{Duhalde}. Fine properties of the stationary distributions of  CBI processes, when they exist, have also been recently established in Chazal et al. \cite{zbMATH06817592} and Keller-Ressel and Mijatovic \cite{KELLERRESSEL20122329}.  In the case where no stationary distribution exists, less attention has been paid to the limit theorems of CBI processes. It will certainly not be surprising that the results found in the seventies for Galton-Watson processes with immigration, have counterparts in the continuous-state and continuous-time framework. A year after Kawazu and Watanabe's founding work, Pinsky  thus published  a short note \cite{Pinsky} without proof, on the limits of CBIs. We believe it is of interest to write down some details and resume in this article the study of limit theorems for CBIs initiated by Pinsky. 

We start by proving an almost-sure convergence for CBI processes by adapting Grey's approach \cite{Grey} to the framework with immigration (Theorem \ref{theorem1}). We then  provide a  general non-linear renormalization in law (Theorem \ref{no1}).  To the best of our knowledge this latter renormalization does not appear in the literature about Galton-Watson processes with immigration. We explain now our main results. Denote respectively by $\Psi$ and $\Phi$ the branching and immigration mechanisms, we will recall their definitions in the next section.  In the case of supercritical branching, we show the existence of two distinct almost-sure asymptotic regimes according to the convergence/divergence of the integral $\int_0 \frac{\Phi(u)}{|\Psi(u)|}\ddr u$. 
When this integral converges, i.e $\int_0 \frac{\Phi(u)}{|\Psi(u)|}\ddr u<\infty$, the branching dynamics takes precedence over immigration and directs the divergence of the process towards infinity. More precisely, under the classical $L\ln L$ moment assumption (also called Kesten-Stigum condition) over the branching L\'evy measure, the CBI process grows at the same exponential rate  as the pure branching process.  
On the other hand, when it diverges, i.e $\int_0 \frac{\Phi(u)}{|\Psi(u)|}\ddr u=\infty$, immigration is so substantial that the branching, although supercritical, is somehow overtaken. So that, typically the process grows faster than the pure branching process on the event of its non-extinction. 

A similar dichotomy occurs more generally for non-critical CBI processes when we consider their long-term behaviour in law. Our main contribution is to design a non-linear time-dependent renormalization in law of a non-critical CBI($\Psi,\Phi$)-process $(Y_t,t\geq 0)$ satisfying $\int_0 \frac{\Phi(u)}{|\Psi(u)|}\ddr u=\infty$. We shall find, see Theorem \ref{no1}, a deterministic function $\lambda \mapsto r_t(\lambda)$ only depending on $\Psi$ and $\Phi$ such that 
\begin{equation}\label{nonlinearintro} r_t(1/Y_t)\underset{t\rightarrow  \infty}{\longrightarrow} \mathrm{e}_1 \text{ in  law}
\end{equation}
where $\mathrm{e}_1$ is a standard exponential random variable. The latter renormalization is actually equivalent to the following property, see Corollary \ref{cor}. Given two independent CBI$(\Psi,\Phi)$ processes $(Y_t,t\geq 0)$ and $(\tilde{Y}_t,t\geq 0)$ such that $\int_0 \frac{\Phi(u)}{|\Psi(u)|}\ddr u=\infty$, one has
\begin{equation}\label{extreme}\frac{Y_t}{\tilde{Y}_t}\underset{t\rightarrow  \infty}{\longrightarrow} \Lambda \text{ in  law}
\end{equation}
where $\mathbb{P}(\Lambda=0)=\mathbb{P}(\Lambda=\infty)=\frac{1}{2}$. As a consequence of \eqref{extreme}, we shall see  that no function $(\eta(t),t\geq 0)$ exists such that $\eta(t)Y_t$ converges in law towards a nondegenerate random variable (i.e whose support is not contained in $\{0,\infty\}$). This has been shown by Cohn \cite{Cohn} in the setting of discrete time and space. The meaning of the limit in law \eqref{nonlinearintro} will be made more explicit by introducing further assumptions on the rate of divergence of the integral $\int_0 \frac{\Phi(u)}{|\Psi(u)|}\ddr u$, namely on the speed at which $\int_\epsilon \frac{\Phi(u)}{|\Psi(u)|}\ddr u$ goes to $\infty$ as $\epsilon$ goes to $0$. In the same vein as Pakes, see \cite{Pakes}, we design three regimes of divergence: Slow (S), Log (L) and Fast (F). Each corresponds to a specific renormalization and a specific limiting law. The faster the integral diverges, the more the branching dynamics is overtaken by immigration. This is reflected by the different renormalizations occurring  in the three regimes. In particular, in the fast case (F) the branching mechanism plays no role in the renormalization. Pinsky's result \cite[Theorem 2]{Pinsky} which corresponds to the subcritical case under condition (S) has now a proof, see Remark \ref{pinskyproof}--ii), and a misprint in his statement is corrected. 

\textbf{Notation:} We denote respectively by $\overset{d}{\longrightarrow}$ and $\overset{p}{\longrightarrow}$ the convergence in law and the convergence in probability. We use the relation symbol $\sim$ when the ratio of the two terms on the two sides of it converges to $1$ (if any of the two terms is random, the convergence holds almost surely).  The probability measure and its expectation are denoted by $\mathbb{P}$ and $\mathbb{E}$. For any $x\geq 0$, $\mathbb{P}_x$ denotes the law of a CBI process started from $x$. The integrability of a function $f$ in a neighbourhood of $0$ is denoted by $\int_0 f(x)\ddr x<\infty$ (similarly  $\int^{\infty} f(x)\ddr x<\infty$ denotes the integrability of $f$ in a neighbourhood of $\infty$).  Last, we denote functions, either deterministic or random, vanishing in the limit by $o(1)$. 
 
The paper is organized as follows. First we recall in Section \ref{preliminaries} the definition of a CBI process and some of its most fundamental properties. Our main results are stated in Section \ref{results}. We first establish in Section \ref{aslimit} the almost-sure convergence results in the supercritical case. Then, Section \ref{limitinlaw} is devoted to the study of convergence in law in the non-critical case when $\int_0\frac{\Phi(u)}{|\Psi(u)|}\ddr u=\infty$. We define the three regimes (S), (L), (F) in Section \ref{threeregimes}. The last section treats some critical branching mechanisms fulfilling certain regular variation properties.

\section{Preliminaries}\label{preliminaries}
We recall hereafter the definition of a CBI process and some of its most fundamental properties. Our main references are Chapter 3 of Li's book \cite{Li} and Chapter 12 of Kyprianou's book \cite{Kyprianoubook}. We say that a random variable is  nondegenerate if its support is not contained in $\{0,\infty\}$ and proper if it is finite almost surely. 

Write $\pi$ and $\nu$ for two $\sigma$-finite nonnegative measures on $(0,\infty)$ satisfying respectively  $\int^\infty_0 (z\wedge
z^2) \,\pi(\d z)<\infty$ and $\int^\infty_0 (1\wedge
z) \,\nu(\d z)<\infty$. Consider a triple $(\sigma,b, \beta)$ such that $\sigma\geq 0, b\in \mathbb{R}$ and $\beta\geq 0$. Let $\Psi$ be the Laplace exponent of a spectrally positive L\'{e}vy process with finite mean (here we assume $|\Psi'(0+)|<\infty$, so that in particular the CBI process does not explode) and whose characteristic triple is $(b,\sigma,\pi)$. Let $\Phi$ be the Laplace exponent of a subordinator with drift $\beta$ and L\'evy measure $\nu$. They are specified by the L\'{e}vy-Khinchine formula
\begin{align*}
\Psi(q)&=bq+\frac{1}{2}\sigma^{2}q^{2}+\int_{0}^{\infty}(e^{-qu}-1+qu)\pi(\ddr u),\quad q\geq0,
\end{align*}
So $\Psi$ is convex (i.e., $\Psi''(q)\geq 0, \forall q\geq 0$) with $\Psi(0)=0$. Similarly, 
\begin{align*}
\Phi(q)&=\beta q+\int_{0}^{\infty}(1-e^{-qu})\nu(\ddr u),\quad q\geq0,
\end{align*}
So $\Phi$ is a concave continuous,  strictly increasing function with $\Phi(0)=0$. 

A CBI process with branching and immigration mechanisms $\Psi$ and $\Phi$, is a strong Markov process  $(Y_{t}, t\geq 0)$ taking values in   $[0,\infty)$ whose transition kernels are characterized by their Laplace transforms.  So for $\lambda\geq 0$ and $x\in \mathbb{R}_{+}$,  \begin{equation}\label{cumulantCBI}
\mathbb{E}_x[e^{-\lambda Y_{t}}]=\exp \left(-xv_{t}(\lambda)-\int_{0}^{t}\Phi(v_{s}(\lambda))\ddr s\right),
\end{equation}
where the map $t\mapsto v_{t}(\lambda)$ is the solution to the differential equation
\begin{equation}\label{cumulant} \frac{\partial}{\partial t}v_{t}(\lambda)=-\Psi(v_{t}(\lambda)), \quad v_{0}(\lambda)=\lambda. \end{equation}
Note that $v_{t+s}(\lambda)=v_t(v_s(\lambda))$ from the Markov property. 

Existence and unicity of CBI processes have been established in \cite[Theorem 1.1]{KAW}.
Recently Dawson and Li \cite{DawsonLi}, see also \cite{jiao20}, have shown that any CBI is the strong solution of a certain stochastic differential equation (SDE) with jumps.  

Suppose that $(\Omega,\mathscr{F}_t,\mathbb{P})$ is a filtered probability space satisfying the usual hypotheses. Let $\{B_t\}_{t\ge0}$ be an $(\mcr{F}_t)$-Brownian motion.  Let $N_0(\d s,\d z,\d u)$ and $N_1(\d s, \d u)$ denote two
$(\mcr{F}_t)$-Poisson random measures on $(0,\infty)^3$ and $(0,\infty)^2$  with
intensities $\d s \,\pi(\d z)\,\d u$ and $\d s \,\nu(\d z)$. We assume that the Brownian motion and the Poisson random measures are independent of each other.
Let $\tilde N_0(\d s,\d z,\d
u)$ be the corresponding compensated measure of $N_0$, namely $\tilde N_0(\d s,\d z,\d
u):= N_0(\d s,\d z,\d u)-\d s \,\pi(\d z)\,\d u.$  The following SDE
\begin{equation}\label{sdeB}\begin{split} Y_t = Y_0&+\sigma\int_0^t \sqrt{Y_{s}} \,\d B_s\\
&+\int_0^t\int_{0}^\infty\int_0^{Y_{s-}} z\,\tilde{N}_0(\d
s, \d z, \d u)+\int_0^t(\beta-bY_s)\d s+\int_0^t\int_{0}^\infty zN_1(\d
s, \d z).
\end{split}\end{equation}
admits a unique strong solution whose law is that of  a CBI with branching mechanism $\Psi$ and immigration mechanism $\Phi$. When there is no immigration, that is to say $\Phi \equiv 0$,  the drift $\beta$ and the Poisson random measure $N_1$ vanish and the process $(Y_t,t\geq 0)$ solution to \eqref{sdeB} is a continuous-state branching process (CB process for short) with branching mechanism $\Psi$. When $\Psi\equiv 0$, only the immigration part remains and $(Y_t,t\geq 0)$ is a subordinator with Laplace exponent $\Phi$. Lastly, we recall that a (sub)-critical CB$(\Psi)$ process conditioned on the non-extinction is a CBI$(\Psi,\Phi)$ process with $\Phi=\Psi'-\Psi'(0+)$, see Lambert \cite{Lambert2}, Li \cite[Theorem 4.1]{Li2000} and Fittipaldi and Fontbona \cite{Fontbona}. From now on, unless explicitely mentioned, we consider processes with immigration, namely $\Phi(q)>0$ for all $q>0$.

Recall the form of the function $\Psi$ and notice that $b=\Psi'(0+)$. A CBI process is said to be critical, subcritical or supercritical according as $b=0$, $b>0$ or $b<0$. Note that $\Psi$ has at most two roots. Introduce 
\beqnn
\rho=\inf\{z>0, \Psi(z)\geq 0\}, \quad \inf\varnothing=\infty. \eeqnn
We see that  $\rho=0$ if $b\geq 0$ and $\rho>0$ if $b<0$. In particular $\rho=\infty$ if and only if $-\Psi$ is the Laplace exponent of a subordinator.  By (\ref{cumulant}), if $0<\lambda<\rho$ (resp. $\lambda>\rho$), then $v_{t}(\lambda)\in[\lambda,\rho]$ is increasing (resp. $v_{t}(\lambda)\in[\rho,\lambda]$ is decreasing) in $t$. Then (\ref{cumulant}) implies 
\begin{equation}\label{cumulantintegral} \int_{v_{t}(\lambda)}^{\lambda}\frac{\ddr z}{\Psi(z)}=t, \quad\forall t\in[0,\infty),\quad \forall \lambda\in(0,\infty)/\{\rho\},
\end{equation}
Recall $\bar{v}_t:=\underset{\lambda \rightarrow \infty}{\lim \uparrow} v_{t}(\lambda)\in [0, \infty]$ and set
$\bar{v}:=\underset{t \rightarrow \infty}{\lim \downarrow} \bar{v}_{t}\in [0, \infty].$
Grey shows in \cite{Grey} that \begin{equation}\label{Grey}
\bar{v}_t<\infty \text{ for all } t>0 \text{ if and only if } \int^{\infty}\frac{\ddr q}{\Psi(q)}<\infty \text{ (Grey's condition)}.
\end{equation}
Note that $\rho\leq \bar{v}$, and if $\bar{v}<\infty$ then $\bar{v}=\rho$.

Recall \eqref{cumulantCBI}. Define the map $r_{t}(\lambda):=\int_{0}^{t}\Phi(v_{s}(\lambda))\ddr s$. A simple change of variable gives
\begin{equation}\label{J_lambda t}r_{t}(\lambda)=\begin{cases} \!\!\!\!\!\!\!&\int_{v_t(\lambda)}^{\lambda}\frac{\Phi(u)}{\Psi(u)}\ddr u \text{ if } \Psi \not\equiv 0 \\\
\!\!\!\!\!\!\!&t\Phi(\lambda) \text{ if } \Psi \equiv 0.
\end{cases}
\end{equation}

Then \eqref{cumulantCBI} can also be written as 
 \begin{equation}\label{cumulantCBI2}
\mathbb{E}_x[e^{-\lambda Y_{t}}]=\exp \left(-xv_{t}(\lambda)-r_t(\lambda) \right). 
\end{equation}
Note also that for any $t\geq 0$ and any $n\in \mathbb{N}$, $Y_t=Y_t^{1}+\cdots+Y_t^{n}$ in law where $((Y_t^{i})_{t\geq 0}, 1\leq i\leq n)$ are i.i.d copies of a CBI$(\Psi,\frac{1}{n}\Phi)$ process. So that in particular, $Y_t$ has an infinite divisible law on $\mathbb{R}_+$ and $\lambda\mapsto r_t(\lambda)$ is the Laplace exponent of a subordinator (with no killing term). For any $t\geq 0$, we set $r_t(\infty)=\int_{\bar{v}_t}^{\infty}\frac{\Phi(u)}{\Psi(u)}\ddr u\in (0,\infty]$, where $\bar{v}_t:=\underset{\lambda \rightarrow \infty}{\lim} \uparrow v_t(\lambda)$, with the convention that if $\int^{\infty}\frac{\Phi(u)}{\Psi(u)}\ddr u=\infty$ then $r_t(\infty)=\infty$ for all $t>0$. From \eqref{J_lambda t}, we easily check that $r_t(\infty)<\infty$ as soon as $\int^{\infty}\frac{\Phi(u)}{\Psi(u)}\ddr u<\infty$. Letting $\lambda$ tend to $\infty$ in \eqref{cumulantCBI2} readily entails that $r_t(\infty)<\infty$ if and only if $\mathbb{P}_x(Y_t=0)>0$. We refer the reader interested in the zero-set of CBIs to \cite{MR3263091}.

We refer to Section 3.2 of \cite{Li} for proofs of the following technical statements; see also \cite{Grey}. We gather in the next lemma analytical results on the map $\lambda\mapsto v_{t}(\lambda)$ and its inverse (whenever it exists).
\begin{lemma}\label{eta}
The map $\lambda\mapsto v_{t}(\lambda)$ is  strictly increasing on $[0,\infty)$.  For any $t\geq 0$, let $\lambda\mapsto v_{-t}(\lambda)$ be the inverse map of $\lambda\mapsto v_t(\lambda)$. This is a strictly increasing function, well-defined on $[0,\bar{v}_{t})$ which satisfies for all $s, t\geq 0$ and $0\leq \lambda<\bar{v}_{s+t}$
$$v_{-(s+t)}(\lambda)=v_{-s}(v_{-t}(\lambda)).$$
For $0\leq \lambda<\bar{v}_{t}$, such that $\Psi(\lambda)\neq 0$, by  (\ref{cumulantintegral}) one has
\begin{equation}\label{equationeta} \int_{\lambda}^{v_{-t}(\lambda)}\frac{\ddr z}{\Psi(z)}=\int_{v_{t}(v_{-t}(\lambda))}^{v_{-t}(\lambda)}\frac{\ddr z}{\Psi(z)}=t.
\end{equation}
In particular, in the supercritical case, i.e  $b\in(-\infty, 0)$, for $\lambda\in (0,\rho)$
\beqlb\label{eta integral}
\frac{\partial v_{-t}(\lambda)}{\partial t}=\Psi(v_{-t}(\lambda)), \quad v_{0}(\lambda)=\lambda.
\eeqlb
The map $t\mapsto v_{-t}(\lambda)$ is decreasing  and by letting $t\rightarrow\infty$ in (\ref{equationeta}) we see that $v_{-t}(\lambda)\underset{t\rightarrow \infty}{\longrightarrow} 0$. 
Moreover, 
 $v_{-(t+u)}(\lambda)/v_{-t}(\lambda)\underset{t\rightarrow \infty}{\longrightarrow} e^{bu}$ for any $u\geq0$.
\end{lemma}
The following theorem has been announced by Pinsky \cite{Pinsky} and provides some first information on the growth rate. It has been established in the (sub)critical case by Li, see \cite[Theorem 3.20, P.66]{Li} and Keller-Ressel and Mijatovi\'{c} see \cite[Appendix]{KELLERRESSEL20122329}. 
\begin{theorem}[Pinsky \cite{Pinsky}, Li \cite{Li}]\label{PL}
Let $(Y_{t}, t\geq 0)$ be a CBI process with $|\Psi'(0+)|<\infty$. Set $\tau(t)=e^{bt}$ if $b<0$ and $\tau(t)=1$ if $b\geq 0$. The process $(\tau(t)Y_{t}, t\geq 0)$ converges in law, as $t\rightarrow \infty$, towards a proper random variable if and only if 
\begin{equation*}\label{LD}
\int_{0}\frac{\Phi(u)}{|\Psi(u)|}\ddr u<\infty.
\end{equation*}
If $\int_{0}\frac{\Phi(u)}{|\Psi(u)|}\ddr u=\infty$, then for all $z\geq 0$, $\mathbb{P}_x(\tau(t)Y_{t}\leq z)\underset{t\rightarrow \infty }\longrightarrow 0$, that is to say $(\tau(t)Y_{t}, t\geq 0)$ converges to $\infty$ in probability.
\end{theorem}
Our paper aims to enrich the above Theorem \ref{PL} by studying almost-sure limits in the  supercritical case and finding new results on the growth rates when $\Psi$ is non-critical and $\int_{0}\frac{\Phi(u)}{|\Psi(u)|}\ddr u=\infty$. It is clear that a supercritical CBI process $(Y_t,t\geq 0)$ is transient i.e. $Y_t\underset{t\rightarrow \infty}{\longrightarrow} \infty$ a.s. The properties of transience and recurrence for subcritical and critical CBI processes have been studied in Duhalde et al. \cite{Duhalde}. It is established in \cite[Theorem 3]{Duhalde} that a (sub)critical CBI$(\Psi,\Phi)$ process is recurrent or transient according as
\beqlb\label{transient}
\mathcal{E}:=\int_0 \frac{\ddr x}{\Psi(x)}\exp\left(-\int_{x}^{1}\frac{\Phi(u)}{\Psi(u)}\ddr u\right)=\infty\ \mbox{or}\  <\infty.
\eeqlb
We see that the integral $\int_{0}\frac{\Phi(u)}{|\Psi(u)|}\ddr u$ plays a crucial role in this integral test. In particular, it is worth noticing that in the (sub)critical case, the divergence  of $\int_{0}\frac{\Phi(u)}{|\Psi(u)|}\ddr u$ is necessary for the CBI process to be transient but not sufficient. 
%
%
%
%

The notion of regularly varying functions will be used in several places. Recall that a function $R$ is  regularly varying at $\infty$ (resp. at $0$) with index $\theta\in \mathbb{R}$ if for any $\lambda>0$
\begin{equation*}\label{RV}\frac{R(\lambda x)}{R(x)}\rightarrow\lambda^{\theta}\text{ as } x\rightarrow \infty \,(\text{resp. } 0).\end{equation*}
The function $R$ is said to be \textit{slowly varying} if  $\theta=0$ and if $R$ is regularly varying with index $\theta$, then $R$ has the form $R(x)=x^{\theta}L(x)$ for all $x\geq 0$ with $L$ a slowly varying function. We stress that those functions occur naturally in the study as for instance, in the supercritical case, i.e. $b\in(-\infty,0)$, Lemma \ref{eta} ensures that the function $t\mapsto v_{-\ln(t)}(\lambda)$ is regularly varying at $\infty$ with index $b$. 
We refer the reader to Bingham et al. \cite{Bingham87} for a reference on those functions.
\section{Results}\label{results}
\subsection{Almost-sure limits}\label{aslimit}
This section deals with the so-called Seneta-Heyde norming for CBI processes. We refer to Seneta \cite{MR0270460} and Heyde \cite{heyde1970} for the seminal papers in the discrete setting; see also Lambert \cite{Lambert2}. When no immigration is taken into account, namely $\Phi\equiv 0$, this study has been done by Grey \cite{Grey} and Bingham \cite{BINGHAM1976217}. We refer the reader for instance to the end of Chapter 12 of Kyprianou's book \cite{Kyprianoubook}. 

Recall $\rho \in [0,\infty]$ the largest root of $\Psi$,  the map $t\mapsto v_{-t}(\lambda)$ and its equation \eqref{eta integral}.
\begin{theorem}\label{theorem1} Let $(Y_{t}, t\geq 0)$ be a CBI$(\Psi, \Phi)$ with a supercritical branching mechanism $\Psi$ i.e $b<0$. Let $0<\lambda<{{\rho}}$. Then,
\begin{itemize}
\item[i)] if $\int_{0}\frac{\Phi(u)}{|\Psi(u)|}\ddr u<\infty$ then $v_{-t}(\lambda)Y_t\underset{t\rightarrow \infty}{\longrightarrow} W^{\lambda}$ $\mathbb{P}_x$-a.s., 
where $W^{\lambda}$ is a non-degenerate proper random variable with Laplace exponent 
\begin{equation}\label{WLK}
\mathbb{E}_x\left[e^{-\theta W^{\lambda}}\right]=\exp\left(-xv_{-\frac{\ln\theta}{b}}(\lambda)+\int_{0}^{v_{-\frac{\ln\theta}{b}}(\lambda)}\frac{\Phi(u)}{\Psi(u)}\ddr u\right),
\end{equation}
\item[ii)] if $\int_{0}\frac{\Phi(u)}{|\Psi(u)|}\ddr u=\infty$ then $v_{-t}(\lambda)Y_t\underset{t\rightarrow \infty}{\longrightarrow} \infty$ $\mathbb{P}_x$-a.s.
\end{itemize}
\end{theorem}
\begin{remark}\label{rem1} \begin{itemize} \item[i)] It is worth mentioning that Grey \cite[Theorem 2]{Grey} found the same almost-sure renormalization $v_{-t}(\lambda)$ for the supercritical process without immigration, namely $\Phi\equiv 0$, on the event of non-extinction. See also Duquesne and Labb\'e \cite[Lemma 2.2]{DuqLab} for the expression of the Laplace transform \eqref{WLK} with $\Phi\equiv 0$. As explained in the introduction, Theorem \ref{theorem1} reflects the fact that two regimes occur according to the convergence/divergence of the integral $\int_{0}\frac{\Phi(u)}{|\Psi(u)|}\ddr u$. In case i), the branching dictates the growth of the process, in case ii) it is overwhelmed by the immigration.  Moreover, when $\int^\infty (x\ln x) \pi(\ddr x)<\infty$, we have $v_{-t}(\lambda)\underset{t\rightarrow \infty}{\sim} K_{\lambda}e^{bt}$ for some constant $K_\lambda>0$. So in this case, as mentioned in the introduction, the CBI process grows almost-surely exponentially fast.
\item[ii)] In the non-critical case, $|\Psi(u)|\underset{u\rightarrow 0}{\sim} |b|u$  and the condition $\int_0\frac{\Phi(u)}{|\Psi(u)|}\ddr u<\infty$ is equivalent to $\int^{\infty}\ln(u)\nu(\ddr u)=\infty$ where $\nu$ is the immigration measure $\nu$, see the forthcoming Remark \ref{logassumption}. 
\end{itemize}
\end{remark}

\begin{proof}
The proof follows from a simple adaptation of Grey's martingales, see \cite{Grey}, to the setting of CBIs. Consider $(Y_{t}, t\geq 0)$ a CBI$(\Psi, \Phi)$. Recall \eqref{cumulantCBI2}. Fix $\lambda \in (0, \rho)$. For every $t\geq 0$, set $\kappa_{\lambda}(t):=\exp\left(\int_{\lambda}^{v_{-t}(\lambda)}\frac{\Phi(u)}{\Psi(u)}\ddr u\right)$.  We show that the process $(M_t^{\lambda}, t\geq 0)$ defined for any $t\geq 0$ by $M^{\lambda}_{t}=\kappa_{\lambda}(t)\exp\left(-v_{-t}(\lambda)Y_{t}\right)$ is a positive martingale. The random variables $M^{\lambda}_{t}$ are plainly integrable and for any $s,t\geq 0$,
\begin{align*}
\E[M_{t+s}^{\lambda}|\FF_t]&=\kappa_{\lambda}(t+s)\E[\exp(-v_{-(t+s)}(\lambda)Y_{t+s})\lvert Y_t] \quad \text{\it (by the Markov property)}&\\
&=\kappa_{\lambda}(t+s)\exp\left(-Y_tv_s(v_{-(t+s)}(\lambda))-\int_{v_s(v_{-(t+s)}(\lambda))}^{v_{-(t+s)}(\lambda)}\frac{\Phi(u)}{\Psi(u)}\ddr u\right)&\\
&=\kappa_{\lambda}(t+s)\exp\left(\int_{v_{-(t+s)}(\lambda)}^{v_{-t}(\lambda)}\frac{\Phi(u)}{\Psi(u)}\ddr u\right)\exp\big(-Y_tv_{-t}(\lambda)\big) &\\
&=\kappa_{\lambda}(t)\exp\big(-Y_tv_{-t}(\lambda)\big)=M_t^{\lambda},&
\end{align*}
where the second equality follows from (\ref{cumulantCBI2}) and the third from the fact that $ v_s(v_{-(t+s)}(\lambda))=v_s(v_{-s}\circ v_{-t}(\lambda))=v_{-t}(\lambda)$. 
In particular, we see that the process $(e^{-v_{-t}(\lambda)Y_{t}(x)}, t\geq 0)$ is a positive supermartingale. This entails that the process $(
v_{-t}(\lambda)Y_{t}, t\geq 0)$ converges, as $t$ goes to infinity, $\mathbb{P}_x$-almost-surely in $\bar{\mathbb{R}}_{+}{:=[0,\infty]}$. Denote its limit by $W^{\lambda}$. We shall see that it is infinite almost-surely in case ii).

Since $\Psi$ is supercritical, one has $\rho>0$,  and $\Psi(u)<0$ for $0<u<\rho$. By Lemma \ref{eta}, we know that when $0<\lambda<\rho$, 
$v_{-t}(\lambda)\underset{t\rightarrow \infty}\longrightarrow 0$.
For fixed $\theta> 0$, we choose $t$ large enough such that $v_{-t}(\lambda), \theta v_{-t}(\lambda)\in (0,\rho)$. 
Recall that by assumption $b\in (-\infty,0)$. Then
{it can be established from Equations \eqref{cumulant} and \eqref{equationeta}, see e.g. \cite[Lemma 2.2]{DuqLab} and \cite[Theorem 3.13]{Li} for details, that 
\begin{equation}\label{kappaconv}v_t(\theta v_{-t}(\lambda))\underset{t\rightarrow \infty}\longrightarrow v_{-\frac{\ln \theta}{b}}(\lambda).\end{equation} 
Note that the above convergence holds no matter {{the sign of $-\frac{\ln \theta}{b}$}}.
}
By \eqref{cumulantCBI2}, one has for all $t\geq 0$
\begin{equation*}\label{laplaceW} \mathbb{E}_x[e^{-\theta v_{-t}(\lambda)Y_{t}}]=\exp\left(-xv_t(\theta v_{-t}(\lambda))-r_t(\theta v_{-t}(\lambda))\right)
\end{equation*}
and by letting $t$ to $\infty$ in the right-hand side above we get the expression \eqref{WLK} for the Laplace transform of $W^{\lambda}$. It is not hard to see that $v_{-\frac{\ln \theta}{b}}(\lambda)\in(0,\rho)$ as $\int_0\frac{\ddr z}{\Psi(z)}=-\infty$ and $\int_\lambda^\rho\frac{\ddr z}{\Psi(z)}=-\infty$. In fact, if $\rho<\infty$, then $\Psi(\rho)=0$; if $\rho=\infty$, then $-\Psi$ is the Laplace exponent of a subordinator which implies that $\int^{\infty}\frac{\ddr z}{\Psi(z)}=-\infty$.
Moreover by  (\ref{equationeta}), 
\begin{equation}\label{kappa}v_{-\frac{\ln \theta}{b}}(\lambda)\underset{\theta \rightarrow 0}\longrightarrow 0\text{ and }
v_{-\frac{\ln \theta}{b}}(\lambda)\underset{\theta \rightarrow \infty}\longrightarrow \rho.\end{equation}
By using the first convergence in \eqref{kappa}, we observe that if $\int_{0}\frac{\Phi(u)}{|\Psi(u)|}\ddr u<\infty$, then $W^{\lambda}<\infty$ $\mathbb{P}_x$-almost-surely. Applying the second convergence in \eqref{kappa} and the fact that $\int^{\rho}\frac{\ddr u}{\Psi(u)}=-\infty$, we get
$\mathbb{P}(W^{\lambda}_x=0)=\exp\left(-x\rho+\int_{0}^{\rho}\frac{\Phi(u)}{\Psi(u)}\ddr u\right)=0.$
If $\int_{0}\frac{\Phi(u)}{|\Psi(u)|}\ddr u=\infty$, then $\int_{0}\frac{\Phi(u)}{\Psi(u)}\ddr u=-\infty$ as $\Psi(u)<0$ for $0<u<\rho$. We get from by (\ref{WLK}) that $\mathbb{E}_x[e^{-\theta W^{\lambda}}]=0$  for any $\theta>0$, therefore $W^{\lambda}=\infty$ almost-surely.
\qed
\end{proof}

The next theorem sheds some light on what limit theorems can be expected in the case $\int_{0}\frac{\Phi(u)}{|\Psi(u)|}\ddr u=\infty$. 

\begin{theorem}\label{noaslimit} Let $(Y_t,t\geq 0)$ be a supercritical CBI$(\Psi,\Phi)$. Assume $\int_{0}\frac{\Phi(u)}{|\Psi(u)|}\ddr u=\infty$. 
For any positive deterministic function  $(\eta(t),t\geq 0)$, if
$(\eta(t)Y_t,t\geq 0)$ converges almost surely then its limit is either $0$ almost surely or $\infty$ almost surely.
\end{theorem}
\begin{remark} 
Theorem \ref{noaslimit} breaks the hope of finding any law of large numbers and can be seen as a starting point for a study through non-linear renormalisations, see the forthcoming Section \ref{limitinlaw}.
\end{remark}
\begin{proof} We shall use the framework of flow of SDEs as Dawson and Li \cite{DawsonLi}.
We recall that by replacing in the SDE \eqref{sdeB}, the Brownian motion $(B_t,t\geq 0)$ by a white noise $M(\ddr s,\ddr u)$, see \cite{DawsonLi} and \cite{ChunhuaLi} for details, one can consider on a same probability space, the SDEs  
\begin{equation}\label{flow}\begin{split} Y^{(n)}_t(x)= x&+\sigma\int_n^{n+t} \int_0^{Y^{(n)}_{s}(x)} \, M(\d s,\d u)+\int_n^{n+t}\int_{0}^\infty\int_0^{Y^{(n)}_{s-}(x)} z\,\tilde{N}_0(\d
s, \d z, \d u) \\
&+\int_n^{n+t}(\beta-bY^{(n)}_s(x))\d s+\int_n^{n+t}\int_{0}^\infty zN_1(\d
s, \d z).
\end{split}\end{equation}
Those SDEs are known to have pathwise unique solutions. More precisely this provides a sequence of flows of CBI$(\Psi,\Phi)$ processes $\{Y^{(n)}_t(x), x\geq0, t\geq0, n\geq 0\}$ such that for any $n\in \mathbb{N}$ and any $x\geq y\geq 0$, $(Y^{(n)}_t(x)-Y^{(n)}_t(y),t\geq 0)$ is a CB process started from $x-y$ with branching mechanism $\Psi$ and is independent of $\{Y_t^{(n)}(y),t\geq0\}$. We denote by $(Y_t(x),t\geq 0)$ the process $(Y_t^{(0)}(x),t\geq 0)$, the solution to \eqref{flow} for $n=0$. Pathwise uniqueness entails that $Y_t^{(n)}(Y_n(x))=Y_{n+t}(x)$ for any $x\geq 0$, $t\geq 0$ and $n\in \mathbb{N}$ almost-surely. Let $X_t^{(n)}(x)=Y_t^{(n)}(x)-Y_t^{(n)}(0)$. Then 
\beqnn
Y_{n+t}(0)=X^{(n)}_{t}(Y_n(0))+Y^{(n)}_{t}(0).
\eeqnn
By Theorem \ref{theorem1}-(ii), $v_{-t}(\lambda)Y^{(n)}_t(0)\underset{t\rightarrow \infty}{\longrightarrow} \infty$ a.s and applying Theorem \ref{theorem1}-(i) to the CB$(\Psi)$ process $(X^{(n)}_{t}(Y_n(0)),t\geq 0)$, see Remark \ref{rem1}-i), we get $v_{-t}(\lambda)X^{(n)}_{t}(Y_n(0))\underset{t\rightarrow \infty}{\longrightarrow} W^{\lambda}$ for some finite random variable $W^{\lambda}$. Hence  
\beqlb\label{ratio}
\frac{Y_{n+t}(0)}{Y^{(n)}_t(0)}=1+\frac{X^{(n)}_{t}(Y_n(0))}{Y^{(n)}_t(0)}\underset{t\rightarrow \infty}{\longrightarrow} 1\quad\mbox{a.s..}
\eeqlb
Assume that there exists some $\eta(t)>0$ such that $\eta(t)Y_t(0)\underset{t\rightarrow \infty}{\longrightarrow} V_0$ a.s.. Then by (\ref{ratio}),  $\eta(n+t)Y^{(n)}_t(0)\underset{t\rightarrow \infty}{\longrightarrow} V_0$ a.s., and for $\ell\in\mathbb{N}$,  $\eta(n+l)Y^{(n)}_{\ell}(0)\underset{\ell\rightarrow \infty}{\longrightarrow} V_0$ a.s.. However by iteration,  it is not hard to see that for $\ell\geq1$, 
\beqnn
Y^{(n)}_{\ell}(0)&=& X_{\ell-1}^{(n+1)}(Y_1^{(n)}(0))+Y_{\ell-1}^{(n+1)}(0)\\
&=&\sum_{k=1}^{\ell} X_{\ell-k}^{(n+k)}(Y_1^{(n+k-1)}(0)),\quad \big(Y_0^{(n+\ell)}(0)=0\big)
\eeqnn
where $\{X_{\cdot}^{(k)}(Y_1^{(k-1)}(0))\}_{k=1}^\infty$  is a sequence of independent CB($\Psi$) processes. 

By the above iteration, for fixed $n$,  $\{Y^{(n)}_{\ell}(0),\ell\geq1\}$ is  measurable with respect to the $\sigma$-algebra $\mathcal{G}_n$ generated by  the sequence of independent processes $\{X_t^{(k)}(Y_1^{(k-1)}(0)): t\geq0\}_{k=n}^\infty$.  Since for any $n$, $\eta(n+l)Y^{(n)}_{\ell}(0)\underset{\ell\rightarrow \infty}{\longrightarrow} V_0$ a.s., we immediately have that $V_0$ is  measurable with respect to the tail $\sigma$-algebra generated by  the same sequence of independent processes $\{X_t^{(k)}(Y_1^{(k-1)}(0)): t\geq0\}_{k=1}^\infty$, i.e. $\cap_{n=1}^\infty \mathcal{G}_n$.


Kolmogorov's zero-one law, see e.g. \cite[Theorem 2.5.1]{Durrett}, asserts that  $V_0$ is a constant or infinite a.s..   Assume that $V_0$ is a finite positive constant a.s.. Since $Y^{(n)}_{\cdot}(0)$ has the same distribution as $Y_{\cdot}(0)$, 
we immediately have that $\eta(t)Y^{(n)}_t(0)\underset{t\rightarrow \infty}{\longrightarrow} V_0$ a.s.. Then $\eta(n+t)/\eta(t)\underset{t\rightarrow \infty}{\longrightarrow}1$, which implies that $\eta(t)\sim L(e^t)$ for some slowly varying function $L$ at $\infty$.  However recalling that  $v_{-\ln t}$ is a regularly varying function with index $b$, see Lemma \ref{eta}, we have that $v_{-t}(\lambda)\sim e^{bt}L^*(e^{t})$ as $t$ goes to $\infty$ where $L^*(\cdot)$ is a slowly varying function at $\infty$ and thus $v_{-t}(\lambda)/\eta(t)\rightarrow 0$. 
This leads to a contradiction. Thus
$V_0$ is $0$ a.s. or $\infty$ a.s..
\qed
\end{proof}
\subsection{A general limit in law for non-critical CBIs when $\int_0\frac{\Phi(u)}{|\Psi(u)|}\ddr u=\infty$}\label{limitinlaw}
This section focuses on the study of the long-term behavior of CBI$(\Psi,\Phi)$ processes satisfying the condition
\begin{equation}\label{cinfty}\int_{0}\frac{\Phi(u)}{|\Psi(u)|}\ddr u=\infty. \end{equation}
By Theorem \ref{PL}, in this case $Y_t$ converges in law to $\infty$ as $t$ goes to $\infty$. We give a finer description of the behavior of $Y_t$ through distributional, rather than almost sure, limit theorems. In this section, we prove the main convergence theorem below which provides a non-linear time-dependent renormalization in law of any non-critical CBI process. The three different regimes of convergence in law mentioned in the Introduction are designed in the forthcoming Section \ref{threeregimes}. 

Recall the definition of $r_t(\lambda)$ in \eqref{J_lambda t} and that $r_t(\infty)=\infty$ if and only if $\mathbb{P}_x(Y_t=0)=0$ for all $t>0$. In the next theorem, we take the convention $1/0=\infty$.
\begin{theorem}\label{no1}
Assume {(\ref{cinfty}) holds}, and
$\Psi$ is non-critical ($b\neq 0$). Then, for all $x\geq 0$, we have
\begin{equation}\label{newe}
r_t(1/Y_t):=\int_{v_{t}(1/Y_t)}^{1/Y_t} \frac{\Phi(u)}{\Psi(u)}\ddr u\stackrel{d}\longrightarrow \mathrm{e}_1, \text{ as } t\rightarrow \infty
 \text{ under } \mathbb{P}_{x}
\end{equation}
where $\mathrm{e}_1$ is an exponential random variable with parameter $1$.
\end{theorem}
\begin{remark} The law of the limiting distribution in \eqref{newe} does not depend on the initial state $x$ of the CBI process. It justifies the perception that in the regime \eqref{cinfty}, the dynamics is governed on the long-term by the immigration part and not by the branching part. 
\end{remark}
\begin{proof}
Recall  the equations \eqref{cumulantCBI} and \eqref{cumulantCBI2}. 


Step 1: We claim that for fixed $\lambda>0$, $r_t(\lambda)\rightarrow\infty$ as $t\rightarrow\infty$. In fact,
in the subcritical case, $\Psi(u)>0$ for $u>0$. By (\ref{cumulantintegral}), $v_t(\lambda)\downarrow 0$ as $t\to\infty$, for any fixed $\lambda>0$. According to (\ref{cinfty}) and the second equality of \eqref{J_lambda t}, we have that $r_{t}(\lambda)\rightarrow\infty$ as $t\to\infty.$ 

In the supercritical case,
 still by (\ref{cumulantintegral}), $v_t(\lambda)\rightarrow \rho$ as $t\to\infty$ for fixed $\lambda>0$. Then $\Phi(v_t(\lambda))\rightarrow\Phi(\rho)>0$,
 as $t\rightarrow\infty$.
Together with the first equality of \eqref{J_lambda t}, we obtain that $r_t(\lambda)\rightarrow\infty$ as $t\rightarrow\infty$.

Step 2: Recall that $\lambda\mapsto r_{t}(\lambda)$ is the Laplace exponent of a subordinator with no killing term. For all $t$, $r_t(0)=0$, and $r_t$ is strictly decreasing  in $\lambda$. 
So we can define $\lambda\mapsto c_{t}(\lambda)$ as the inverse of $\lambda \mapsto r_{t}(\lambda)$. 
Fix $\lambda>0$. By Step 1, for any small $\varepsilon>0$, we can find sufficiently large $t$ such that $r_t(\varepsilon)>\lambda=r_{t}(c_{t}(\lambda))$, which implies that
$c_t(\lambda)<\varepsilon$.
Thus
\begin{equation}\label{00}c_t(\lambda)\underset{t\rightarrow\infty}{\rightarrow} 0,\text{ and } v_t(c_t(\lambda))\underset{t\rightarrow\infty}{\rightarrow} 0.\end{equation} 
The second limit follows from (\ref{cinfty}) and the second equality of (\ref{J_lambda t}) (replacing $\lambda$ by $c_t(\lambda)$).


Step 3:
Notice that we can equivalently show that for any $\lambda\geq0$ and $\theta>0$,
\beqlb\label{limit Y_t}
\lim_{t\to\infty}\E_x[e^{-\theta c_t(\lambda)Y_t}]=e^{-\lambda}.
\eeqlb
In fact, if (\ref{limit Y_t}) holds, then $c_t(\lambda)Y_t$ converges in distribution to a random variable $Z$ such that $\P(Z=\infty)=1-\P(Z=0)=1-e^{-\lambda}.$
Therefore, for $\lambda>0$,
\begin{equation}\label{key}\mathbb{P}(r_{t}(1/Y_t)>\lambda)=\mathbb{P}(1/Y_{t}> c_{t}(\lambda))=\mathbb{P}(c_{t}(\lambda)Y_t<1)\underset{t\rightarrow \infty}{\longrightarrow} e^{-\lambda},
\end{equation}
which implies \eqref{newe}.

Step 4: Note that
\begin{equation}\label{decompose}\E_x[e^{-\theta c_t(\lambda)Y_t}]=\exp{\left(-xv_{t}(\theta c_t(\lambda))-r_t(\theta c_t(\lambda))\right)}\end{equation}
and
\begin{align}\label{rthetact}
r_t(\theta c_t(\lambda))&=\int_{v_t(\theta c_t(\lambda))}^{\theta c_t(\lambda)}\frac{\Phi(u)}{\Psi(u)}\ddr u&\\
&={\int_{c_t(\lambda)}^{\theta c_t(\lambda)}\frac{\Phi(u)}{\Psi(u)}\ddr u}+\underbrace{\int_{v_{t}(c_t(\lambda))}^{c_t(\lambda)}\frac{\Phi(u)}{\Psi(u)}\ddr u}_{=r_{t}(c_t(\lambda))=\lambda}+{\int_{v_{t}(\theta c_t(\lambda))}^{v_{t}(c_t(\lambda))}\frac{\Phi(u)}{\Psi(u)}\ddr u.}&\nonumber
\end{align}
So to obtain (\ref{limit Y_t}), it suffices to prove that as $t\rightarrow \infty$, 
\begin{equation}\label{vthetac}v_t(\theta c_t(\lambda))\rightarrow0\end{equation}
and 
\beqlb \label{limit super 2}
\int_{c_t(\lambda)}^{\theta {c_t(\lambda)}}\frac{\Phi(u)}{\Psi(u)}\ddr u\rightarrow0,\quad  \int_{v_t(c_t(\lambda))}^{v_t(\theta c_t(\lambda))}\frac{\Phi(u)}{\Psi(u)}\ddr u\rightarrow0.\eeqlb
 By the monotonicity of $\Phi$, we have for small enough $y$,
\beqlb\label{3.22}
\left|\int_{\theta y}^{y}\frac{\Phi(u)}{\Psi(u)}\ddr u\right|\leq \max(\Phi(y),\Phi(\theta y))\left|\int_{\theta y}^{y}\frac{1}{\Psi(u)}\ddr u\right|.
\eeqlb
On the one hand $ \max(\Phi({y}),\Phi(\theta y))\underset{y\rightarrow 0}\longrightarrow 0$ as $\Phi$ is continuous and $\Phi(0)=0$. On the other hand, since $\Psi$ is non-critical, there exists some constant $h>0$ such that $|\Psi(u)|\geq hu$ for $u$ close enough to $0$. This entails that
 \beqlb\label{equa non critical}
 \bigg|\int_{\theta y}^{y}\frac{1}{\Psi(u)}\ddr u\bigg|\leq \frac{|\ln\theta|}{h},
 \eeqlb
when $y$ is small enough. Then by (\ref{3.22}) and (\ref{00}),
\beqlb
\label{limit theta integral}
\bigg|\int_{c_t(\lambda)}^{\theta {c_t(\lambda)}}\frac{\Phi(u)}{\Psi(u)}\ddr u\bigg|\leq \max(\Phi(c_t(\lambda)),\Phi(\theta c_t(\lambda)))\frac{|\ln\theta|}{h}
\underset{t\rightarrow \infty}{\longrightarrow} 0.
\eeqlb
So the first convergence in \eqref{limit super 2} is proved. 

 Note that 
\beqnn
\int_{v_t(c_t(\lambda))}^{v_t(\theta c_t(\lambda))}\frac{\ddr u}{\Psi(u)}\ar=\ar\int_{v_t(c_t(\lambda))}^{c_t(\lambda)}\frac{\ddr u}{\Psi(u)}+
\int_{c_t(\lambda)}^{\theta c_t(\lambda)}\frac{\ddr u}{\Psi(u)}+\int_{\theta c_t(\lambda)}^{   v_t(\theta c_t(\lambda))}\frac{\ddr u}{\Psi(u)}\nonumber\\
\ar=\ar t+\int_{c_t(\lambda)}^{\theta c_t(\lambda)}\frac{\ddr u}{\Psi(u)}-t\nonumber\\
\ar=\ar \int_{c_t(\lambda)}^{\theta c_t(\lambda)}\frac{\ddr u}{\Psi(u)},
\eeqnn
which entails that $v_t(\theta c_t(\lambda))\to 0$ as $t\to\infty$, since $v_t(c_t(r))\to 0$ by (\ref{00}) and $\big|\int_{c_t(\lambda)}^{\theta c_t(\lambda)}\frac{\ddr u}{\Psi(u)}\big|\leq\ln\theta/h$ by (\ref{equa non critical}). So the convergence in \eqref{vthetac} is proved. 
Then
\beqnn
\bigg|\int_{v_t(c_t(\lambda))}^{v_t(\theta c_t(\lambda))}\frac{\Phi(u)}{\Psi(u)}\ddr u\bigg|\ar\leq\ar
\max\{\Phi(v_t(c_t(\lambda))),\Phi(v_t(\theta c_t(\lambda)))\}\left|\int_{v_t(c_t(\lambda))}^{v_t(\theta c_t(\lambda))}\frac{1}{\Psi(u)}\ddr u\right|\nonumber\\
\ar=\ar
\max\{\Phi(v_t(c_t(\lambda))),\Phi(v_t(\theta c_t(\lambda)))\}\left|\int_{c_t(\lambda)}^{\theta c_t(\lambda)}\frac{1}{\Psi(u)}\ddr u\right|,
\eeqnn
which goes to $0$ similarly as in (\ref{limit theta integral}). Then the second convergence in \eqref{limit super 2} is proved. So both conditions \eqref{vthetac} and \eqref{limit super 2} hold true, and we can conclude that (\ref{newe}) holds true. \qed
\end{proof}
\begin{remark}\label{rem3}
Note that the steps 1 and 2 also work for the critical case. But step 4 requires $b\neq 0.$ However the same line of arguments as in this proof will be used in Section \ref{criticalsec} where we focus on the study of the critical case.
\end{remark}
We now provide a corollary leading to a more intuitive probabilistic understanding of Theorem \ref{no1}. In particular it will shed new light on Theorem \ref{noaslimit}. We take the convention $0/0=0\times \infty=0$.
\begin{corollary}\label{cor} Assume \eqref{cinfty} and that $\Psi$ is non-critical. Let $(Y_t,t\geq 0)$ and $(\tilde{Y}_t,t\geq 0)$ be two independent CBI$(\Psi,\Phi)$ processes started from $0$. 

Then 
\begin{equation}\label{tilde}
Y_t/\tilde{Y}_t\underset{t\rightarrow \infty}{\longrightarrow} \Lambda \text{ in law}
\end{equation}
where $\Lambda$ has law $\mathbb{P}(\Lambda=0)=1-\mathbb{P}(\Lambda=\infty)=1/2$.
Moreover, there is no deterministic renormalization $(\eta(t),t\geq 0)$ such that $(\eta(t)Y_t,t\geq 0)$ converges in law towards a non-degenerate random variable.
\end{corollary}



\begin{proof} By a similar argument as in Equation \eqref{key}, since for any $\lambda>0$, $c_t(\lambda)Y_t\underset{t\rightarrow \infty}{\longrightarrow} Z$ in law with $Z$ such that $\mathbb{P}(Z=\infty)=1-\mathbb{P}(Z=0)=1-e^{-\lambda}$, one has for any $\theta>0$ and 
\[\mathbb{P}(r_t(\theta/Y_t)>\lambda)=\mathbb{P}(c_t(\lambda)Y_t<\theta)\underset{t\rightarrow \infty}{\longrightarrow} e^{-\lambda}=\mathbb{P}(\mathrm{e}_1>\lambda)\]
with $\mathrm{e}_1$ a standard exponential random variable. Hence for any $\theta\geq 0$ and for any $t\geq 0$, we apply \eqref{cumulantCBI2} to obtain \begin{align*}
\mathbb{E}\left[e^{-\theta\frac{{Y}_t}{\tilde Y_t}}\right]=\mathbb{\tilde E}_0\left[{\mathbb{E}}_0\left[e^{-\frac{\theta}{\tilde Y_t}{Y}_t}|\tilde Y_t\right]\right]&= \mathbb{\tilde E}\left[\mathbb{E}\left[e^{- r_t(\theta/\tilde Y_t)}|\tilde Y_t\right]\right]\\
&=\mathbb{\tilde E}\left[e^{-r_t(\theta /\tilde Y_t)}\right]\underset{t\rightarrow \infty}{\longrightarrow}\mathbb{E}\left[e^{-\mathrm{e}_1}\right]=\frac{1}{2}.
\end{align*} 

Therefore $Y_t/\tilde{Y}_t\underset{t\rightarrow \infty}{\longrightarrow} \Lambda$ in law with $\mathbb{P}(\Lambda=0)=\frac{1}{2}$ and $\mathbb{P}(\Lambda=\infty)=\frac{1}{2}$. 

We show now that there is no renormalization in law. By contradiction, assume that there exists $(\eta(t),t\geq 0)$ such that $\eta(t)Y_t\underset{t\rightarrow \infty}{\longrightarrow} V$ in law with $V$ a non-degenerate random variable. Let $b>a>0$ be any two values such that $\P(V\in[a,b])>0.$  Then, 
$$\liminf_{t\to\infty}\P\left(\frac{\eta(t){Y}_t}{\eta(t)\tilde{Y}_t}=\frac{{Y}_t}{\tilde{Y}_t}\in[a/b, b/a]\right)>0,$$ 
but this is in contradiction to \eqref{tilde}. 
Then necessarily $V$ is degenerate. \qed
\end{proof}

\begin{remark} The statement of Corollary \ref{cor} holds true for CBI processes started from arbitrary initial values. Indeed if $(Y_t(x),t\geq 0)$ and $(\tilde{Y}_t(y),t\geq 0)$ are two independent CBI$(\Psi,\Phi)$ started respectively at $x$ and $y$, then 
for any $t\geq 0$, $Y_t(x)=X_t(x)+Y_t(0)$ with $(X_t(x),t\geq 0)$ a CB$(\Psi)$ started from $x$ and independent of $(Y_t(0),t\geq 0)$. Similarly $\tilde{Y}_t(y)=\tilde{X}_t(y)+\tilde{Y}_t(0)$ with $(\tilde{X}_t(y),t\geq 0)$ a CB$(\Psi)$ started from $y$ and independent of $(\tilde Y_t(0),t\geq 0)$. One checks that 
\[\frac{Y_t(x)}{\tilde{Y}_t(y)}=\frac{Y_t(0)\left(1+X_t(x)/Y_t(0)\right)}{\tilde{Y}_t(0)\left(1+\tilde{X}_t(y)/\tilde{Y}_t(0)\right)}\overset{p}{\sim} \frac{Y_t(0)}{\tilde{Y}_t(0)} \text{ as } t\rightarrow \infty.\]
Indeed, on the one hand if $\Psi$ is (sub)critical then both $\tilde{X}_t(x)$ and $X_t(y)$ converge towards $0$ almost-surely, and  by Theorem \ref{PL}, $Y_t(0)$ and $\tilde{Y}_t(0)$ go towards $\infty$ in probability. On the other hand if $\Psi$ is supercritical then by Theorem \ref{theorem1}, for any $x\geq 0$,
$\frac{X_t(x)}{Y_t(0)}=\frac{v_{-t}(\lambda)X_t(x)}{v_{-t}(\lambda)Y_t(0)}\underset{t\rightarrow \infty}{\longrightarrow} 0 \text{ a.s.}$
\end{remark}

\subsection{Three different regimes}\label{threeregimes}
Since the renormalization in Theorem \ref{no1} is non-linear and time-dependent, it is rather intricate at a first sight to deduce from it which explicit growth rates are possible. We design here different regimes for which \eqref{cinfty} holds and the rate can be found explicitly. This establishes and completes \cite[Theorem 2]{Pinsky}.
\subsubsection{Definition of the regimes and preliminary calculations}
Recall $\lambda \mapsto c_t(\lambda)$ the inverse of $\lambda\mapsto r_t(\lambda)=\int_{v_t(\lambda)}^{\lambda}\frac{\Phi(u)}{\Psi(u)}\ddr u$. Theorem \ref{no1} indicates that $Y_t$ should grow at the speed of $1/c_t(\lambda)$ as $t\rightarrow \infty$. However the magnitude of $c_t(\lambda)$ is rather involved and deserves a careful analysis. We shall simplify \eqref{newe} to more straightforward forms by imposing some additional conditions. 

To start with, let us fix some constant $\lambda_0$ such that $\lambda_0\in(0,\infty)$ in the (sub)critical case and $\lambda_0\in(0,\rho)$ in the supercritical case. Put
\beqlb\label{defvarphi}
 \varphi(\lambda)=\int_\lambda^{\lambda_0} \frac{\ddr u}{|\Psi(u)|},\quad  0<\lambda<\lambda_0.
 \eeqlb
By assumption $|\Psi'(0+)|<\infty$ and thus $\varphi(\lambda)\rightarrow\infty$ as $\lambda\rightarrow0$.
The mapping $\varphi:(0,\lambda_0)\;\rightarrow\;(0,\infty)$ is { strictly decreasing}, and we write $g$ for its inverse mapping. It is easy to see that $g$ is a strictly decreasing continuous function on $(0,\infty)$, and 
\begin{equation}\label{gx}\lim_{x\to\infty}g(x)=0,\quad \lim_{x\to0}g(x)=\lambda_0.\end{equation} By  (\ref{cumulantintegral}), if $b\geq0$, then $\Psi\geq 0$ and
\begin{align}\label{phisub}
\varphi(v_t(\lambda))&=\int_{v_t(\lambda)}^{\lambda} \frac{\ddr u}{\Psi(u)}+\int_\lambda^{\lambda_0} \frac{\ddr u}{\Psi(u)}=t+\varphi(\lambda).&
\end{align}
{{Similarly if $b<0$, then $\Psi(u)\leq 0$ for $0 \leq u\leq \rho$ and provided that $v_t(\lambda)\in (0,\lambda_0)$
\begin{align}\label{phisuper}
\varphi(v_t(\lambda))&=-\int_{v_t(\lambda)}^{\lambda} \frac{\ddr u}{\Psi(u)}-\int_\lambda^{\lambda_0} \frac{\ddr u}{\Psi(u)}=-t+\varphi(\lambda).&
\end{align}
Applying $g$ to both sides entails that if $b\geq 0$
\beqlb\label{vg}
v_t(\lambda)=g(\varphi(\lambda)+t) ,\quad 0<\lambda<\lambda_0, \  t>0
\eeqlb
and if $b<0$  and { $v_t(\lambda)\in(0,\rho)$}
\beqlb\label{-vg}
v_t(\lambda)=g(\varphi(\lambda)-t) ,\quad 0<\lambda<\lambda_0, \  t>0.
\eeqlb
}}

Then for any $x,y>0$ such that $\Psi$ never attains zero between $x$ and $y$,  we obtain by a change of variable that
\begin{equation}\label{phig}
\int_x^y\frac{\Phi(u)}{|\Psi(u)|}\ddr u=\int_{g(\varphi(x))}^{g(\varphi(y))}\frac{\Phi(u)}{|\Psi(u)|}\ddr u=\int_{\varphi(y)}^{\varphi(x)}\Phi(g(u))\ddr u.
\end{equation}

Inspired by Pinsky \cite[Theorem 2]{Pinsky},
we introduce the following function to characterize the divergence of the integral in (\ref{cinfty}):
\begin{equation}\label{effdrift}
H(x):=\left\{ \begin{array}{l}
\displaystyle\frac{1}{|b|}\int_{e^{-x}}^1\frac{\Phi(u)}{u}\ddr u, \mbox{ if } b\in(-\infty,0)\cup(0,\infty);\\
\\
\displaystyle\int_{g(x)}^{\lambda_0}\frac{\Phi(u)}{|\Psi(u)|}\ddr u, \mbox{ if }b=0 \\
\end{array}\right., \quad x\geq 0
\end{equation}
where $H(0)$ takes the value $\lim_{x\to 0}H(x)=0$ (using \eqref{gx}).
%
It is not hard to check that the condition (\ref{cinfty}) is equivalent to $H(x)\underset{x\rightarrow  \infty}{\longrightarrow}\infty$.
Based on (\ref{defvarphi}), a simple calculation shows that 
$H'$ is strictly decreasing and $H'(x)\rightarrow0$ as $x\rightarrow\infty$. 
We introduce now different regimes of speed of divergence of the function $H$ at $\infty$.
We refer the reader to Bingham et al. \cite{Bingham87} for a reference on those functions. The following three different modes of convergence to $0$ of $H'$ correspond to different possible modes of divergence of $H$:
 \begin{itemize}
\item[(S)] (slow-divergence) 
$xH'(x)\rightarrow0$ as $x\rightarrow\infty$ and $H(x)\rightarrow\infty$;
\item[(L)] (log-divergence) $xH'(x)\rightarrow a$ for some constant $a>0$ as $x\rightarrow\infty$;
\item[(F)] (fast-divergence) $xH'(x)\rightarrow\infty$ as $x\rightarrow\infty$ and $H'$ is regularly varying at $\infty$.
\end{itemize}
Note that conditions (L) and (F) always entail $H(x)\longrightarrow \infty$ as $x$ goes to $\infty$, while the limit $xH'(x)\rightarrow0$ as $x\rightarrow\infty$  in condition (S) does not guarantee it. Since $H'(x)\rightarrow0$ as $x\rightarrow\infty$,
Condition (F) can be given in the following equivalent form:
 \beqlb\label{regular H'}
 H'(x)=x^{-\delta}\frac{1}{ L(x)},
 \eeqlb
where $L$ is slowly varying at $\infty$ and $0\leq\delta\leq 1$, and if $\delta=0$, $L(x)\rightarrow\infty$ as $x\rightarrow\infty$; if $\delta=1$, $L(x)\rightarrow0$ as $x\rightarrow\infty$.

Let $h$ be the inverse map of $1/H'$.
Under  (\ref{regular H'}) with $\delta>0$,  $t\mapsto h(t)$ is regularly varying with index $1/\delta$ at $\infty$, i.e. $h(t)\sim t^{1/\delta}L^*(t)$
for some slowly varying function $L^*(t)$. Moreover, since $H'(x)\rightarrow0$ as $x\rightarrow\infty$, 
 we have 
\begin{equation}\label{ainf}
\lim_{t\to\infty}h(t)=\infty.
\end{equation}

It might be difficult to verify the three conditions. We provide a proposition below 
to study the asymptotic behaviors of $H$ and $H'.$ To this purpose, { recall the immigration mechanism
$$\Phi(q)=\beta q+\int_{0}^{\infty}(1-e^{-qu})\nu(\ddr u).$$ 
As observed by Pinsky \cite[P. 244]{Pinsky}, for a non-critical CBI process, a faster rate of divergence in (\ref{cinfty}) implies heavier tails of the L\'{e}vy measure $\nu(\ddr u)$.  The following result specifies this idea. }
\begin{proposition}\label{H'equiv} Assume that $b\neq 0$ and
$H(x)\rightarrow\infty$ as $x\rightarrow\infty$. Denote by $\bar{\nu}$ the tail of $\nu$: for all $u\geq 0$, $\bar{\nu}(u)=\nu([u,\infty))$. Then 
\beqnn
H(x)\sim\frac{1}{|b|}\int_{1}^{e^{x}}\frac{\bar{\nu}(u)}{u}\ddr u,\quad  \text{ as }x\rightarrow \infty.
\eeqnn
Moreover $H'(x)=\Phi(e^{-x})/|b|$ for any $x\geq 0$ and if $u\mapsto \bar\nu(u)$ is slowly varying at $\infty$  then
\beqnn
H'(x)\sim \bar{\nu}(e^x)/|b|,\quad  \text{ as }x\rightarrow \infty.\eeqnn
\end{proposition}
\begin{remark} The tail $u\mapsto\bar\nu(u)$ is slowly varying at $\infty$ if and only if $\Phi$ is slowly varying at $0$, see \cite[Theorem 8/1/6, P.333]{Bingham87} \end{remark}
\proof 
Note that  $\Phi(u)/u=\beta+\int_0^{\infty} e^{-ut}\bar{\nu}(t)\ddr t$.
A simple calculation shows that for $z\in(0,1)$
\beqnn
\int_{z}^1\frac{\Phi(u)}{u}\ddr u\ar=\ar\beta(1-z)+\int_0^{\infty}\frac{\bar{\nu}(u)}{u}(e^{-zu}-e^{-u})\ddr u\\
\ar=\ar\beta(1-z)+\bigg(\int_0^1+\int_1^{\infty}\bigg)\frac{\bar{\nu}(u)}{u}(e^{-zu}-e^{-u})\ddr u.
\eeqnn
Since $\int_0^1u\nu(\ddr u)<\infty$, 
$\int_1^{\infty} \frac{\bar{\nu}(u)}{u}e^{-u}\ddr u<\infty$ and $\int_{0+}\frac{\Phi(u)}{u}\ddr u=\infty$, we have that as $z\rightarrow0$,
\beqlb\label{prop eq 1}
\int_{z}^1\frac{\Phi(u)}{u}\ddr u\sim \int_1^{\infty}\frac{\bar{\nu}(u)}{u}e^{-zu}\ddr u,
\eeqlb
which implies $\int_1^x \frac{\bar{\nu}(u)}{u}\ddr u\rightarrow\infty$ as $x\rightarrow \infty$.
It is not hard to see that $\int_1^x \frac{\bar{\nu}(u)}{u}\ddr u$ is also slowly varying at $\infty$. It follows from Tauberian theorem (see e.g. Bertoin \cite[P.10]{Ber96})
that
 \beqlb\label{prop eq 2}
 \int_1^x \frac{\bar{\nu}(u)}{u}\ddr u\sim\int_1^{\infty} \frac{\bar{\nu}(u)}{u}e^{-u/x}\ddr u, \quad x\rightarrow\infty.
  \eeqlb
The first result follows from (\ref{prop eq 1}) and (\ref{prop eq 2}). For the second result, note that $H'(x)=\Phi(e^{-x})/|b|$. Without loss of generality, we can assume that the parameter $\beta$ in $\Phi$ equals zero. Applying the Tauberian theorem entails that if $\bar{\nu}(x)\sim \ell(x)$ for some slowly varying function $\ell$ at $\infty$, then $\frac{\Phi(u)}{u}\sim \frac{1}{u}\ell(1/u)$ as $u$ goes $0$. Hence we have 
$H'(x)\sim \bar{\nu}(e^x)/|b|,\quad  x\rightarrow \infty.$
\qed
\begin{remark}\label{logassumption} By letting $z$ to $0$ in \eqref{prop eq 1}, we see that $\int_{0}\frac{\Phi(x)}{x}\ddr x=\infty$ if and only if $\int^{\infty}\frac{\bar{\nu}(u)}{u}\ddr u=\infty$, the latter is equivalent to $\int^{\infty}_1\ln x\  \nu(\ddr x)=\infty$.
\end{remark}
Proposition \ref{H'equiv} allows us to reformulate the three regimes, in the non-critical case, in terms of the tail of the immigration measure $\nu$ when the latter has a slowly varying tail.
 \begin{itemize}
\item[(S)] (slow-divergence) 
$\bar{\nu}(x)\ln x \rightarrow 0$  as $x\rightarrow\infty$ and $\int^{\infty}_1\frac{\bar{\nu}(x)}{x}=\infty$;
\item[(L)] (log-divergence) $\bar{\nu}(x)\ln x \rightarrow c$ for some constant $c>0$ as $x\rightarrow\infty$;
\item[(F)] (fast-divergence) $\bar{\nu}(x)\ln x \rightarrow \infty$ as $x\rightarrow\infty$.
\end{itemize}
The constant $c$ in regime (L) matches with $a|b|$ where $a:=\underset{x\rightarrow  \infty}{\lim} xH'(x)$. We give below some examples of explicit immigration measures $\nu$ for which the three different regimes may occur in the non-critical cases. 
\begin{example}Let $b\in(-\infty,0)\cup(0,\infty)$. 
\begin{enumerate}
\item If 
$\bar{\nu}(x)\sim\frac{1}{\ln x\ln\ln x}$ as $x\rightarrow\infty$,
then $H(x)\sim(\ln\ln x)/|b|$ and $H'(x)\sim1/(|b|x\ln x)$. Condition (S) is satisfied. This example corresponds to Example 3 in \cite[Example 3]{Duhalde} of null-recurrent CBI.
\item If 
$\bar{\nu}(x)\sim c/{\ln x}$ 
for some constant $c>0$, as $x\rightarrow\infty$,
then $H'(x)\sim c/(|b|x)$. Condition (L) is satisfied.
\item If
$\bar{\nu}(x)\sim\frac{\ln\ln x}{(\ln x)^\delta}, \quad(0<\delta\leq1)$  as $x\rightarrow\infty$,
then $H'(x)\sim (x^{-\delta}\ln x)/|b|$. If as $x\rightarrow\infty$, $\bar{\nu}(x)\sim1/(\ln\ln x)$, then $H'(x)\sim1/(|b|\ln x)$.  Both cases satisfy Condition (F).
\end{enumerate}
\end{example}

\begin{remark}\label{r3}
Recall the integral test $\mathcal{E}<\infty$ or $\mathcal{E}=\infty$ for transience and recurrence of (sub)critical CBIs given in \eqref{transient}. Plainly, by a change of variable, $\mathcal{E}=\int^{\infty}e^{-H(x)}\ddr x$. If {(F) holds}, or (L) is satisfied with $a>1$, then $\mathcal{E}<\infty$ and the process is transient. In the case (S), or (L) with $a\leq 1$, $\mathcal{E}=\infty$ and the process is {null-recurrent}. 
\end{remark}

We state now a side result on the growth rate of a subordinator whose Laplace exponent is slowly varying at $0$.
\begin{proposition}\label{aside} Let $(I_t,t\geq 0)$ be a subordinator with Laplace exponent $\Phi$. Assume that $\Phi$ is slowly varying at $0$, then 
\[t\Phi(1/I_t)\overset{d}{\longrightarrow} \mathrm{e}_1 \text{ as } t\to \infty.\]
\end{proposition}
\begin{proof} This is reminiscent to the steps 1 and 2 in the proof of Theorem \ref{no1}. Observe that the stated convergence holds if and only if for all $\lambda>0$ \[\mathbb{P}(t\Phi(1/I_t)>\lambda)=\mathbb{P}(I_t<1/\Phi^{-1}(\lambda/t))=\mathbb{P}(\Phi^{-1}(\lambda/t)I_t<1)\underset{t\rightarrow \infty}{\longrightarrow} e^{-\lambda}.\]
The latter will occur if for any $\theta>0$, $\mathbb{E}[e^{-\theta \Phi^{-1}(\lambda/t)I_t}]\underset{t\rightarrow \infty}{\longrightarrow} e^{-\lambda}$.
Since $\Phi$ is slowly varying at $0$ and $\Phi^{-1}(\lambda/t)\longrightarrow 0$ as $t\to \infty$, we have 
\[\frac{t}{\lambda}\Phi(\theta \Phi^{-1}(\lambda/t))=\frac{\Phi(\theta \Phi^{-1}(\lambda/t))}{\Phi( \Phi^{-1}(\lambda/t))}\underset{t\rightarrow \infty}{\longrightarrow} 1.\]
Therefore
$$\mathbb{E}[e^{-\theta \Phi^{-1}(\lambda/t)I_t}]=e^{-t\Phi(\theta \Phi^{-1}(\lambda/t))}\underset{t\rightarrow \infty}{\longrightarrow} e^{-\lambda}$$
which finishes the proof.\qed
\end{proof}
In the next subsection, we study how the convergence results can be made explicit by combining Theorem \ref{no1} and the three conditions. These results can be seen as the continuous analogues of those in \cite{Pakes}. 


\subsubsection{Non-critical case}

Consider a non-critical CBI process $(Y_t, t\geq 0)$, i.e.\;$b\neq0$. Recall $v_{-t}(\lambda)$ defined by (\ref{eta integral}) and  $\lambda_0$ given in (\ref
{defvarphi}). Set
\begin{equation*}
\rho_t:=\left\{ \begin{array}{l}
1, \mbox{ if } b>0;\\
v_{-t}(\lambda_0), \mbox{ if }b<0. \\
\end{array}\right.
\end{equation*}
From Theorem \ref{PL} and Theorem \ref{theorem1}-(ii), if (\ref{cinfty}) holds, then $\rho_tY_{t}$ converges to $\infty$ at least in probability.   

\begin{theorem}\label{cno1}
Assume that $b\neq0$.

{\rm (i)}
If Condition (S) holds, let $m(x):=\exp(\int_{1/x}^1\frac{\Phi(u)}{\Psi(u)}\ddr u)$ for $x>0$.  Then 

\begin{equation}\label{U}\frac{\ln \rho_tY_t}{t}\overset{p}{\longrightarrow}0\quad \text{ and }\quad
m(\rho_tY_t)/m(e^{|b|t})\stackrel{d}{\longrightarrow} U,  \text{ as }t\to\infty,\end{equation}
where $U$ is uniformly distributed on $[0,1]$.

{\rm (ii)} If Condition (L) holds, then \beqlb\label{U_L}
\frac{\ln \rho_tY_t}{t}\overset{d}{\longrightarrow}|b|U_L,  \text{ as }t\to\infty, 
\eeqlb
where $\mathbb{P}(U_L\leq \lambda)=\left(\frac{\lambda}{1+\lambda}\right)^a, \quad \lambda\geq0$.

{\rm (iii)} If Condition (F) holds (i.e. \eqref{regular H'} holds with $0\leq \delta\leq 1$), then 
$$\frac{\ln Y_t}{t}\overset{p}{\longrightarrow}\infty\quad\mbox{and}\quad t\Phi(1/Y_t)\overset{d}{\longrightarrow}\mathrm{e}_1,  \text{ as }t\to\infty.$$
In particular, if $0<\delta\leq 1$ in (\ref{regular H'}), then  we have 

\begin{equation}\label{subh} h(t)= t^{1/\delta}L^*(t)\quad\text{and}\quad\frac{\ln Y_t}{h(|b|t)}\overset{d}{\longrightarrow}U_F, \text{ as }t\to\infty,\end{equation}
where $L^*$ is some slowly varying function at $\infty$ and $U_F$ follows the extreme distribution given by
$\P(U_F\leq \lambda)=\exp(-1/\lambda^{\delta}), \quad \lambda\geq0$.
\end{theorem}
\begin{remark}\label{pinskyproof}
 i) We observe from Proposition \ref{aside}, that in the fast regime (F), the branching part plays essentially no role in the growth of the non-critical CBI$(\Psi,\Phi)$ process, since it has the same growth rate as the immigration subordinator $(I_t,t\geq 0)$.

ii) The statement (i) of Theorem \ref{cno1} has been given in Pinsky \cite[Theorem 2]{Pinsky} which however contains some incorrectness. The corrected convergence in  \cite[Theorem 2]{Pinsky} reads: 
$$\P_x(m(X_t)/m(e^{ct})\geq u^{-1})\to u, \forall 0< u\leq 1 \text{ as } t\to\infty.$$
\end{remark}
We will first prove the following lemmas.  

\begin{lemma}\label{lemma no2} Assume that $b\neq0$. Then $\varphi(\lambda)\underset{\lambda \rightarrow 0}{\sim}-\frac{1}{|b|}\ln \lambda$, $ -\ln g(x)\underset{x \rightarrow\infty}{\sim} |b|x, 
 $  and 
 \beqnn
 \ln{v_t(\lambda_t)}\sim-b(t+\varphi(\lambda_t)) \quad\mbox{if}\ b>0; \quad \ln{v_t(\rho_t\lambda_t)}\sim\ln \lambda_t\quad\mbox{if}\ b<0,  \eeqnn
 for any $\lambda_t\rightarrow0+$ as $t\rightarrow\infty$.
   \end{lemma}
 \proof Note that $$\varphi(\lambda)=\int_{\lambda}^{\lambda_0}\frac{\ddr u}{|\Psi(u)|}\underset{\lambda \rightarrow 0}{\sim}\int_{\lambda}^{\lambda_0}\frac{\ddr u}{|b|u}\underset{\lambda \rightarrow 0}{\sim}-\frac{1}{|b|}\ln \lambda.$$
Since $g$ is the inverse function of $\varphi,$ we have   
$ -\ln g(x)\underset{x \rightarrow\infty}{\sim} |b|x.
 $ If $b>0$,  we establish the last statement by plugging $x=\varphi(\lambda_t)+t$ in \eqref{vg}. Now we turn to $b<0$. It follows from (\ref{kappaconv}) and (\ref{kappa}) that 
\beqnn
v_t(\rho_t\lambda)\underset{t\rightarrow \infty}\longrightarrow v_{-\frac{\ln\lambda}{b}}(\lambda_0)\quad\mbox{and}\quad v_{-\frac{\ln\lambda}{b}}(\lambda_0)\underset{\lambda \rightarrow 0}\longrightarrow 0.\eeqnn 
Recall that $\lambda\mapsto v_t(\rho_t\lambda)$ is non-decreasing. Recall also that if $b>0$ then $\rho_t \equiv 1$; if $b<0$ then $\lim_{t\to\infty}\rho_t=0.$ Then we have $\rho_t\lambda_t, v_t(\rho_t\lambda_t)\in(0,\lambda_0)$ for sufficiently large $t$. 
Then by (\ref{-vg}),
\beqlb\label{vg<02}v_t(\rho_t\lambda_t)=g(\varphi(\rho_t\lambda_t)-t)
\eeqlb
for large $t$.  Recall that $\rho_t=v_{-t}(\lambda_0)$.
By (\ref{equationeta}), we obtain
\beqnn
\varphi(\rho_t\lambda_t)-t
=\int_{\rho_t\lambda_t}^{\rho_t}(-1/\Psi(u))\ddr u\sim \int_{\rho_t\lambda_t}^{\rho_t}(-1/bu)\ddr u=\frac{\ln \lambda_t}{b}, \quad \mbox{as}\ t\rightarrow\infty.\eeqnn
Putting $x=\varphi(\rho_t\lambda_t)-t$ into $\ln g(x)\sim bx\,(x\rightarrow\infty)$ and using (\ref{vg<02}), we get the last statement for $b<0$. \qed

\begin{lemma}\label{rho_t in supercritical case} If $b<0$, then $\ln\rho_t\sim bt$ as $t\rightarrow\infty$.
\end{lemma}
\proof Recall Lemma \ref{eta} that $\rho_{t+s}/\rho_t\rightarrow e^{bs}$ as $t\rightarrow\infty$. Consider the function
$\rho_{\ln t}$ on $(1,\infty)$. We observe that $\rho_{\ln t}$ is regularly varying with index $b$ at $\infty$.
 By Bingham et al. \cite[Proposition 1.3.6 (i)]{Bingham87},
  $
  \ln \rho_{\ln t}\sim b\ln t\ (t\rightarrow\infty).
  $
\qed



\noindent {\it Proof of Theorem \ref{cno1}} \; { Recall that under the assumption \eqref{cinfty}, the process $(Y_t, t\geq 0)$ goes to infinity in probability. By the Skorokhod representation theorem,  there is a probability space on which is defined a process whose one-dimensional laws  are those of $(Y_{t},t\geq 0)$, which tends to $\infty$ almost surely when $t$ goes towards $\infty$. See \cite[Corollary]{BlackwellDubins} for a continuous-time version of Skorokhod representation theorem and apply it at time $t=\infty$. In the supercritical case, we apply the Skorokhod representation theorem to the bivariate process $(Y_{t},\rho_{t}Y_{t})_{t\geq 0}$, so that on a certain probability space, copies of both coordinates go to infinity almost surely. We stress that our aim is to establish some convergences in law. One can equivalently work along an arbitrary sequence $(t_n)_{n\geq 1}$ which tends to $\infty$ and apply usual Skorokhod representation theorem. 
}

By the definition of $H$ and Theorem \ref{no1},
 {\begin{equation}\label{equi}
\sign b\Big[H(-\ln v_t(1/Y_t))-H(\ln Y_t)\Big]=(1+o(1))\int_{v_t(1/Y_t)}^{1/Y_t}\frac{\Phi(u)}{\Psi(u)}\ddr u \overset{d}{\longrightarrow}\mathrm{e}_1.
 \end{equation}}
(i) Applying first the mean value theorem, and then Condition (S),   we see that there exists $\theta_t$ between $\ln Y_t$ and $-\ln v_t(1/Y_t)$ such that
 \begin{eqnarray}\label{hh}
 H(-\ln v_t(1/Y_t))-H(\ln Y_t)&=&\int_{\ln Y_t}^{-\ln v_t(1/Y_t)}\frac{uH'(u)}{u}\ddr u\nonumber\\
&=&H'(\theta_t)\theta_t\ln\bigg(-\frac{\ln v_t(1/Y_t)}{\ln Y_t}\bigg)\nonumber\\
&=& o(1)\ln\bigg(-\frac{\ln v_t(1/Y_t)}{\ln Y_t}\bigg).
  \end{eqnarray}
We start with $b>0$. In this case, $-\frac{\ln v_t(1/Y_t)}{\ln Y_t}\geq1$ for sufficiently large $t$. Then a comparison between (\ref{equi}) and (\ref{hh}) entails that $-\frac{\ln v_t(1/Y_t)}{\ln Y_t}\overset{p}\rightarrow+\infty$. It follows from Lemma \ref{lemma no2} with $\lambda_t=1/Y_t$ that 
\beqnn
\frac{\ln Y_t}{t}\overset{p}{\rightarrow}0\quad\mbox{and}\quad-\ln v_t(1/Y_t)\overset{p}{\sim}bt,\quad t\rightarrow\infty.
\eeqnn
Based on the results above, using again Condition (S), there exists $\theta'_t$ between $bt$ and $-\ln v_t(1/Y_t)$ such that
\beqlb\label{m2h}
 H(-\ln v_t(1/Y_t))-H(bt)\ar=\ar H'(\theta'_t)\theta'_t\ln\Big(-\frac{\ln v_t(1/Y_t)}{bt}\Big)\overset{p}{\rightarrow}0,\quad t\rightarrow\infty.
   \eeqlb
Note that 
\begin{eqnarray*}
\frac{m(Y_t)}{m(e^{bt})}={\exp\left(\int^{e^{-bt}}_{1/Y_t}\frac{\Phi(u)}{\Psi(u)}\ddr u\right)}
={\exp\{(H(\ln Y_t)-H(bt))(1+o(1))\}}.\end{eqnarray*}
Also together with (\ref{equi}) and (\ref{m2h}), we have that $\frac{m(Y_t)}{m(e^{bt})}\overset{d}{\rightarrow}\exp(-\mathrm{e}_1)$. We now turn to
$b<0$. In this case, $0<-\frac{\ln v_t(1/Y_t)}{\ln Y_t}\leq1$ for sufficiently large $t$. A comparison between (\ref{equi}) and (\ref{hh}) entails that $-\frac{\ln v_t(1/Y_t)}{\ln Y_t}\overset{p}\rightarrow0$. It follows from Lemma \ref{lemma no2} with $\lambda_t=1/(\rho_tY_t)$ and Lemma \ref{rho_t in supercritical case} that 
\beqnn
-\ln v_t(1/Y_t)\overset{p}{\sim}\ln\rho_tY_t\quad\mbox{and}\quad\ln Y_t\overset{p}{\sim}-bt,\quad t\rightarrow\infty.
\eeqnn
Proceeding as in the case with $b>0$,  we have that $\frac{m(\rho_tY_t)}{m(e^{-bt})}\overset{d}{\rightarrow}\exp(-\mathrm{e}_1)$.

(ii) By Condition (L),
\beqnn
H(-\ln v_t(1/Y_t))-H(\ln Y_t)=\int_{\ln Y_t}^{-\ln v_t(1/Y_t)}H'(u)\ddr u\sim a\int ^{-\ln v_t(1/Y_t)}_{\ln Y_t}\frac{\ddr u}{u},\quad t\rightarrow\infty.\eeqnn
Using Lemma \ref{lemma no2} with $\lambda_t=1/Y_t$ if $b>0$, and using
Lemma \ref{lemma no2} with $\lambda_t=1/(\rho_tY_t)$ and Lemma \ref{rho_t in supercritical case} if $b<0$, we obtain
\beqnn
\sign b\int ^{-\ln v_t(1/Y_t)}_{\ln Y_t}\frac{\ddr u}{u}=
\sign b\ln\bigg(-\frac{\ln v_t(1/Y_t)}{\ln Y_t}\bigg)=\ln \bigg(1+\frac{|b|t}{\ln \rho_tY_t}(1+o(1))\bigg)+o(1).
\eeqnn
Then by (\ref{equi}),
we have Theorem \ref{cno1} (ii).

(iii) {The mean value theorem for integrals shows that there is some $\bar\theta_t$ between $\ln Y_t$ and $-\ln v_t(1/Y_t)$ such that
 \begin{eqnarray*}
 H(-\ln v_t(1/Y_t))-H(\ln Y_t)=H'(\bar\theta_t)\bar\theta_t\ln\bigg(-\frac{\ln v_t(1/Y_t)}{\ln Y_t}\bigg).
  \end{eqnarray*}
By Condition (F) and (\ref{equi}),
we have that
$
-\ln v_t(1/Y_t)\overset{p}{\sim}\ln Y_t.
$ Then it follows from Lemma \ref{lemma no2} (and together with Lemma \ref{rho_t in supercritical case} if $b<0$) that 
 \begin{equation}\label{blny}t/\ln Y_t\overset{p}{\rightarrow}0,\quad\mbox{and}\quad t/\varphi(1/Y_t)\overset{p}{\rightarrow}0,\quad t\rightarrow\infty.\end{equation}
Applying (\ref{vg}), (\ref{-vg}) and \eqref{phig}, we obtain for large $t$, 
\beqlb\label{integral transform}
\int_{v_t(1/Y_t)}^{1/Y_t}\frac{\Phi(u)}{|\Psi(u)|}\ddr u=
\left\{ \begin{array}{l}
\displaystyle\int^{\varphi(1/Y_t)+t}_{\varphi(1/Y_t)}\Phi(g(u))\ddr u, \mbox{ if } b>0;\\
\\
\displaystyle \int^{\varphi(1/Y_t)-t}_{\varphi(1/Y_t)}\Phi(g(u))\ddr u, \mbox{ if } b<0.
\end{array}\right.
\eeqlb
 Note that $\Phi(g(u))=bH'(-\ln g(u))$. 
 
 By the fact that $ -\ln g(u)\sim |b|u$ as $u\rightarrow\infty$ and Condition (F), we have that $\Phi(g(u))$ is regularly varying at $\infty$; see \cite[Proposition 0.8-(iv)]{Res87}.} In the remaining proof, we only deal with $b>0$. The proof for $b<0$ is quite similar and therefore omitted.

By changing variable and applying a mean value theorem to the right-hand side of (\ref{integral transform}), 
\beqlb\label{similar proof}
  \int_{v_t(1/Y_t)}^{1/Y_t}\frac{\Phi(u)}{\Psi(u)}\ddr u
\ar=\ar\varphi(1/Y_t)\int^{t/\varphi(1/Y_t)+1}_{1}\Phi(g(\varphi(1/Y_t)u)) \ddr u\nonumber\\
\ar=\ar t\Phi(g(\varphi(1/Y_t){\hat\theta}_t))   \quad\quad\big(\text{ for some } {\hat\theta}_t\in(1,t/\varphi(1/Y_t)+1)\big)\nonumber\\
\ar\sim\ar t\Phi(g(\varphi(1/Y_t))) \quad (t\rightarrow\infty).
\eeqlb
The validity of the last equivalence is proved as follows: using \eqref{blny}, we have 
$\theta_t\overset{p}{\to}1$ as $t\rightarrow\infty.$ Since $\Phi \circ g$ is regularly varying, the last equivalence  holds by locally uniform convergence, see e.g. \cite[Theorem 1.2.1]{Bingham87}.  Then $t\Phi(1/Y_t)= t\Phi(g(\varphi(1/Y_t)))\overset{d}{\rightarrow}\mathrm{e}_1$ from Theorem \ref{no1}.

Now we focus on the case when $\delta>0$ in (\ref{regular H'}). Note that $H'(x)=\frac{1}{b}\Phi(e^{-x})$. 
Since $h$ is the inverse function of $1/H'$, we have $H'(h(x))=1/x$,  for any $x>b/\Phi(1).$  Then by Karamata's theorem (\cite[P.23]{Res87}), the statement on $h(t)$ in \eqref{subh} holds true.

 For $\lambda>0$, we use again $H'(x)=\frac{1}{b}\Phi(e^{-x})$, and apply (\ref{regular H'}) and \eqref{ainf} to obtain that
\beqnn
\frac{H'(\ln Y_t)}{H'(h(bt)\lambda)} =t\Phi(1/Y_t) \frac{H'(h(bt))}{H'(h(bt)\lambda)}
\overset{d}{\rightarrow}{\lambda^\delta}\mathrm{e}_1 \text{ as }t\to\infty.
 \eeqnn
 Hence,
$ \P(\ln Y_t/h(bt)\leq\lambda)=\P(H'(\ln Y_t))/H'(h(bt)\lambda)\geq1)\underset{t\rightarrow \infty}\longrightarrow\exp(-1/\lambda^\delta)$.
 \qed

 \subsection{On the critical case}\label{criticalsec}
The study  of the critical case, i.e. $b=\Psi'(0+)=0$, is more involved as $v_t(\lambda)$ may have different asymptotics as $t\rightarrow \infty$ and $\lambda\rightarrow 0$ according to the behavior of $\Psi$ near $0$. We make the following assumption on the L\'{e}vy measure $\pi$ in the branching mechanism: 

Suppose that $\pi$ satisfies 
  \beqlb\label{regular varying levy measure}
  \bar{\pi}(u)\underset{u\rightarrow\infty}{\sim} -\frac{1}{\Gamma(-\alpha)}u^{-1-\alpha}\ell(u),
  \eeqlb
where $\bar{\pi}(u)=\pi(u,\infty)$ for $u>0$, $0<\alpha<1$ and $\ell$ is slowly varying at $\infty$. By \cite[Theorem 8.1.6]{Bingham87}, the above assumption is equivalent to
 \beqlb\label{regularly varying psi}
 \Psi(\lambda)\underset{\lambda \rightarrow 0}{\sim} \lambda^{1+\alpha}\ell(1/\lambda).
 \eeqlb
Recall that $\varphi$ is defined by (\ref{defvarphi}) and $g$ is  inverse function of $\varphi$. It follows from Karamata's theorem (\cite[P.17 and P.23]{Res87})
that 
\beqlb\label{regular varying varphi}
\varphi(\lambda)\sim\frac{\lambda^{-\alpha}}{\alpha \ell(1/\lambda)}\quad{and}\quad g(1/\lambda)\sim\lambda^{1/\alpha}\ell^*(1/\lambda), \quad \text{ as }\lambda\rightarrow0,
\eeqlb
where $\ell^*$ is slowly varying at $\infty$. We denote by $\Phi^{-1}$ the inverse function of $\Phi$.

\begin{theorem}\label{ccritical}
Assume that $b=0$ and $(\ref{regular varying levy measure})$ holds. 

{\rm (i)}
If  Condition (S) holds, let $m(x)$  be defined as in \eqref{U}. Then
\beqnn
\frac{m(Y_t)}{m(1/g(t))}\stackrel{d}{\longrightarrow} V \text{ as }t\to\infty,
\eeqnn
where $ V \text{ is uniformly distributed on } [0,1].$

{\rm (ii)} If Condition (L) holds, then 
$$g(t)Y_t\overset{d}{\longrightarrow}V_L \text{ as }t\to\infty,$$
 where $\mathbb{E}[e^{-\lambda V_L}]=(1+\lambda^\alpha)^{-a}, \forall\lambda\geq 0.$

{\rm (iii)}  If Condition (F) holds with $\delta>0$ in (\ref{regular H'}),  then 
\begin{equation}\label{cvfastcritical} \varrho_tY_t\overset{d}{\longrightarrow}V_F \text{ as }t\to\infty \end{equation}
where $\mathbb{E}[e^{-\theta V_F}]=\exp(-\theta^{\delta\alpha}), \text{ for all } \theta\geq 0$, 
with $\varrho_t=\Phi^{-1}(1/t)= t^{-1/(\delta\alpha)}\bar{\ell}(t) \text{ as }t\to\infty$ for some slowly varying function $\bar{\ell}$ at $\infty$. 

If Condition (F) holds with $\delta=0$ in (\ref{regular H'}), then $t\Phi(1/Y_t)\overset{d}{\longrightarrow}\mathrm{e}_1 \text{ as }t\to\infty$.

\end{theorem}

\begin{remark}
When $\delta\in (0,1]$, the convergence \eqref{cvfastcritical} is equivalent to the following $t\Phi(1/Y_t)\overset{d}{\longrightarrow}V_F^{-\delta\alpha}$ as $t$ goes to $\infty$. Indeed for any $\lambda>0$,  the property of regular variation of $\varrho_t$ implies that
\begin{align*}
&\;\;\P(t\Phi(1/Y_t)\geq \lambda)\\
&=\P(\Phi(1/Y_t)\geq \lambda/t)=\P(1/Y_t\geq\varrho_{t/\lambda})&\\
&=\P(\varrho_{t/\lambda}Y_t\leq 1)\sim \P(\lambda^{1/(\delta\alpha)}\varrho_tY_t\leq 1)\underset{t \to \infty}\to\P(V_F\leq \lambda^{-1/(\delta\alpha)})=\P((V_F)^{-\delta\alpha}\geq \lambda).&
\end{align*}
Since the random variable $V_F$ has a stable law, $(V_F)^{-\delta\alpha}$ is not a standard exponential random variable. Therefore, unlike in the non-critical cases, see Theorem \ref{cno1}-(iii)  for which 
$t\Phi(1/Y_t)\overset{d}{\longrightarrow}\mathrm{e}_1$
holds true when condition (F) holds, no matter what value $\delta\in [0,1]$ takes, in the critical case, the convergence to $\mathrm{e}_1$ is only true for $\delta=0.$  
\end{remark}
\proof
We start with some observations. By (\ref{vg}) and \eqref{phig}, we have
 \beqlb\label{changing variables}
 \int_{v_t(\lambda)}^\lambda\frac{\Phi(u)}{\Psi(u)}\ddr u=\int^{t+\varphi(\lambda)}_{\varphi(\lambda)}\Phi(g(u))\ddr u,\quad H(x)=\int^x_{\varphi(\lambda_0)}\Phi(g(u))\ddr u.
  \eeqlb
Then $H'(x)=\Phi(g(x)).$ Note that we will use $H'(x)$ or $\Phi(g(x))$ in different contexts. 

It was mentioned in Remark \ref{rem3} that the steps 1 and 2 in the proof of Theorem \ref{no1} still hold for the critical case. That is to say; recalling  $r_t(\lambda)$ defined by (\ref{J_lambda t}) and $\lambda\mapsto c_{t}(\lambda)$ its inverse; we have for any fixed $\lambda>0$, $c_t(\lambda){\rightarrow} 0$ and $v_t(c_t(\lambda)){\rightarrow} 0$ as $t\rightarrow\infty$. Therefore $\varphi(c_t(\lambda))\to\infty \text{ as } t\to\infty$. 

(i) By (\ref{changing variables}) with $\lambda$ replaced by $c_t(\lambda)$, applying  the mean value theorem,  there exists $\hat\theta_t$ between $\varphi(c_t(\lambda))$ and $t+\varphi(c_t(\lambda))$ such that
\beqlb\label{condition s 2}
\lambda=\int_{v_t(c_t(\lambda))}^{c_t(\lambda)}\frac{\Phi(u)}{\Psi(u)}\ddr u
=\int^{t+\varphi(c_t(\lambda))}_{\varphi(c_t(\lambda))}\Phi (g (u))\ddr u=\hat\theta_t H'(\hat\theta_t){\ln} \Big(\frac{t+\varphi(c_t(\lambda))}{\varphi(c_t(\lambda))}\Big).
\eeqlb
Together  with Condition (S), we have  that $\varphi(c_t(\lambda))/t\rightarrow0 \text{ as } t\rightarrow\infty.$ Then still by the mean value theorem and Condition (S), we obtain
\begin{equation}\label{decom}
\int^{v_t(\theta c_t(\lambda))}_{v_t(c_t(\lambda))}\frac{\Phi(u)}{\Psi(u)}\ddr u
=\int^{t+\varphi(c_t(\lambda))}_{t+\varphi(\theta c_t(\lambda))}\Phi \circ g (u)\ddr u=o(1)\ln \Big(\frac{t+\varphi( c_t(\lambda))}{t+\varphi(\theta c_t(\theta\lambda))}\Big),
\end{equation}
which goes to $0$ as $\varphi$ is regularly varying with index $-\alpha$ by (\ref{regular varying varphi}). Similarly $\int_{\theta c_t(\lambda)}^{c_t(\lambda)}\frac{\Phi(u)}{\Psi(u)}\ddr u\rightarrow0$. Hence by (\ref{rthetact}), (\ref{decompose}) and (\ref{key}), 
\beqlb\label{e_1 2}
\int_{v_{t}(1/Y_t)}^{1/Y_t} \frac{\Phi(u)}{\Psi(u)}\ddr u=\int^{t+\varphi(1/Y_t)}_{\varphi(1/Y_t)}\Phi (g(u)) \ddr u\stackrel{d}\longrightarrow \mathrm{e}_1\text{ as } t\rightarrow\infty.
\eeqlb
Applying the same transformation as in \eqref{decom} to (\ref{e_1 2}) entails that $\varphi(1/Y_t)/t\overset{p}{\rightarrow}0\text{ as } t\rightarrow\infty$. Using again Condition (S),
\beqlb\label{e_1 2'}
\int_{g(t)}^{v_t(1/Y_t)}\frac{\Phi(u)}{\Psi(u)}\ddr u=\int_{t+\varphi(1/Y_t)}^{t}\Phi (g(u)) \ddr u
=o(1)\ln \Big(\frac{t}{t+\varphi(1/Y_t)}\Big)\overset{p}{\rightarrow}0.\eeqlb
Therefore combining the above two displays \eqref{e_1 2} and \eqref{e_1 2'} yields
$$-\ln\Big(\frac{m(Y_t)}{m(1/g(t))}\Big)=\int_{g(t)}^{1/Y_t}\frac{\Phi(u)}{\Psi(u)}\ddr u\overset{d}{\longrightarrow}\mathrm{e}_1.$$
This  allows one to conclude that (i) holds true. 

(ii) By (\ref{cumulantCBI2}) and \eqref{phig}
 \beqnn
 \E_x[e^{-\lambda g(t)Y_t}]= \exp\Big\{-xv_t(g(t)\lambda)-\int^{t+\varphi(\lambda g(t))}_{\varphi(\lambda g(t))}\Phi (g(u))\ddr u\Big\}.
   \eeqnn
  We shall study the two terms in the exponential one by one.  As $t\rightarrow\infty$, $g(t)\to0$ and consequently $v_t(g(t)\lambda)\rightarrow0$. By Condition (L),
\beqnn
\int^{t+\varphi(\lambda g(t))}_{\varphi(\lambda g(t))}\Phi \circ g(u)\ddr u\sim a\int^{t+\varphi(\lambda g(t))}_{\varphi(\lambda g(t))}\frac{\ddr x}{x}=a\ln
\Big(\frac{\varphi(g(t))+\varphi(\lambda g(t))}{\varphi(\lambda g(t))}\Big),
\eeqnn
which converges to $a\ln(\lambda^{\alpha}+1)$ by (\ref{regular varying varphi}). So the statement in (ii) holds true. 

(iii) Applying the same proof as in \eqref{condition s 2},
Condition (F) implies that 
\begin{equation}\label{tphi}t/\varphi(c_t(\lambda))\rightarrow0 \text{ as }  t\rightarrow\infty.\end{equation}  
Since $\phi$ is regularly varying with index $-\alpha$, see (\ref{regular varying varphi}), $\phi(\theta c_t(\lambda))/\phi(c_t(\lambda))\rightarrow \theta^{-\alpha}$
for $\theta>0.$ Then $t/\varphi(\theta c_t(\lambda))\rightarrow0$. 

We also note that  $H'=\Phi \circ g$ is regularly varying by (\ref{regular H'}). Based on the above resulsts, a similar proof as in (\ref{similar proof}) shows that for any $\theta>0$, 
\beqlb\label{asymptotic H'}
\int_{v_{t}(\theta c_t(\lambda))}^{\theta c_t(\lambda)} \frac{\Phi(u)}{\Psi(u)}\ddr u\sim t \Phi\circ g(\varphi(\theta c_t(\lambda)))=t\Phi(\theta c_t(\lambda)),\text{ as }t\rightarrow\infty.
\eeqlb
 Letting $\theta=1$ and $\lambda=1$, by the definition of $r_t(1),c_t(1)$, the left term in the equivalence relation above equals $1$ and so we have 
\begin{equation}\label{tphig}t\Phi(c_t(1))\rightarrow1 \text{ as }  t\rightarrow\infty.\end{equation}

If (\ref{regular H'}) holds, then $\Phi=H'\circ \varphi$ is regularly varying with index $\delta\alpha$ at $0$ by (\ref{regular varying varphi}) and \cite[Proposition 0.8-(iv)]{Res87}. Using the above two displays, 
\beqnn
\int_{v_{t}(\theta c_t(1))}^{\theta c_t(1)} \frac{\Phi(u)}{\Psi(u)}\ddr u\sim t \Phi(c_t(1))\frac{\Phi(\theta c_t(1))}
{\Phi(c_t(1))}\sim \theta^{\delta\alpha}, \quad \text{ as }t\to\infty.
\eeqnn
As $c_t(1)\to 0$, we also have $v_{t}(\theta c_t(1))\to 0$ thanks to the above display. Then using \eqref{decompose} and \eqref{rthetact}, we conclude that
$\E_x[e^{-\theta c_t(1) Y_t}]\rightarrow e^{-\theta^{\delta\alpha}}$, as $t\to\infty.$
 By \cite[Proposition 0.8-(v)]{Res87}, the map $\Phi^{-1}$ is regularly
varying with index $1/(\delta\alpha)$ at $0$, and by \eqref{tphig}, we have that 
$c_t(1)=\Phi^{-1}(\Phi(c_t(1)))\sim\Phi^{-1}(1/t)$. Thus the first result in (iii) holds true.

If $\Phi \circ g$ is slowly varying at $\infty$, then $\Phi \circ g (\varphi(u))$ is slowly varying at $0$ by (\ref{regular varying varphi}). Then it follows from  (\ref{asymptotic H'}) that $\E_x[e^{-\theta c_t(\lambda) Y_t}]\rightarrow e^{-\lambda}$ as $t\rightarrow\infty$.
{ As in (\ref{key}),  we have (\ref{newe}) in this case}. The second half of (iii) follows the very similar proof of Theorem \ref{cno1}-(iii) and we omit it.\qed

\begin{remark} Recall that in the critical case when $\Phi=\Psi'$, the CBI$(\Psi,\Phi)$ process has the same law as the CB($\Psi$) processes conditioned on the non-extinction. Moreover, we readily check that $\int_0\frac{\Psi'(u)}{\Psi(u)}\ddr u=\infty$. It follows that $H'(x)=\Psi'(g(x))$ for all $x$. 
\end{remark}

We apply now Theorem \ref{PL} and Theorem \ref{ccritical} to the case of stable branching and immigration mechanisms for which explicit calculations can be done. 

\begin{corollary}[Stable case] Assume $\Psi(q)=dq^{\alpha+1}$ for $d>0$ and $\alpha\in (0,1]$ and $\Phi(q)=d'q^{\beta}$ for $d'>0$ and $\beta\in (0,1]$. Then,
\begin{itemize}
\item[i)]  if $\beta/\alpha>1$ then $Y_t\overset{d}{\longrightarrow} Y_\infty  \text{ as } t\rightarrow \infty$ where $Y_\infty$ has Laplace transform $\mathbb{E}[e^{-\lambda Y_\infty}]=e^{-\frac{\lambda^{\beta-\alpha}}{\beta-\alpha}}$ for any $\lambda \geq 0$,
\item[ii)]  if $\beta/\alpha=1$ then $t^{-\frac{1}{\alpha}}Y_t\overset{d}{\longrightarrow} (\alpha d)^{\frac{1}{\alpha}}V_L \text{ as } t\rightarrow \infty$ where $V_L$ has Laplace transform $\mathbb{E}[e^{-\lambda V_L}]=\frac{1}{(1+\lambda^{\alpha})^{\frac{d'}{\alpha d}}}$ for any $\lambda\geq 0$,
\item[iii)] if $\beta/\alpha<1$ then $t^{-\frac{1}{\beta}}Y_t\overset{d}{\longrightarrow} (d')^{1/\beta}V_F
\text{ as } t\rightarrow \infty,$ where $V_F$ has Laplace transform $\mathbb{E}[e^{-\lambda V_F}]=e^{-\lambda^{\beta}}$ for any $\lambda\geq 0$.
\end{itemize}
\end{corollary}
\begin{remark} When $\beta=\alpha$, the specific case $d'=(\alpha+1)d$ which corresponds to $\Phi=\Psi'$, has been studied by Kyprianou and Pardo \cite[Lemma 3]{KyprianouPardo} with other techniques. 
\end{remark}
\begin{proof}
First notice that $\int_{0}\frac{\Phi(u)}{\Psi(u)}\ddr u =\infty$ iff $\frac{\beta}{\alpha}\leq 1$. Statement i) is a direct consequence of Theorem \ref{PL}. Let $\lambda_0=1$, we compute  $\varphi(x)=\frac{1}{\alpha d}\left(\frac{1}{x^{\alpha}}-1\right)$ and $g(x)=\left(\frac{1}{\alpha dx+1}\right)^{\frac{1}{\alpha}}$ for $x>0$. Moreover $x\Phi(g(x))\sim \frac{d'}{(\alpha d)^{\beta/\alpha}}x^{1-\beta/\alpha}$ as $x$ goes to $\infty$. If $\beta/\alpha=1$ then Condition (L) is fulfilled for $a=\frac{d'}{\alpha d}$ and we deduce statement ii) from Theorem \ref{ccritical}. If $\beta/\alpha<1$ then Condition (F) and \eqref{regular H'} are fulfilled with $\delta=\frac{\beta}{\alpha}$. Theorem \ref{ccritical} also applies. \qed
\end{proof}

We have focused our study on continuous-state branching processes with immigration whose branching L\'evy measure has finite mean, i.e.\  $\Psi'(0+)=b\in \R$ or equivalently $\int^{\infty}_1 z\pi(\ddr z)<\infty$. The proofs of Theorem \ref{no1} and Corollary \ref{cor} however do not make use of this assumption. They hold then also true in the case of a non-explosive supercritical CBI$(\Psi,\Phi)$ with $\Psi'(0+)=-\infty$. A similar dichotomy occurs in the long-term behavior whether $\int_{0}\frac{\Phi(u)}{|\Psi(u)|}\ddr u<\infty$ or $\int_{0}\frac{\Phi(u)}{|\Psi(u)|}\ddr u=\infty$ when $b=-\infty$. In the first case, some almost-sure non-linear renormalizations can be found, see \cite{Grey2} and Foucart and Ma \cite{foucart2019} for the case without immigration. In the case $\int_{0}\frac{\Phi(u)}{|\Psi(u)|}\ddr u=\infty$, similar regimes than those found in Section \ref{threeregimes} exist, see \cite{BarbourPakes} for the case of discrete state-space processes.

\section*{Acknowledgements} 
The authors thank the anonymous referees for helpful comments. C.F is supported by the French National Research Agency (ANR): LABEX MME-DII (ANR11-LBX-0023-01). C.M is supported by the NSFC of China (11871032) and GXNSF (2018GXNSFAA050031). L.Y is supported by the NSFC of China (11801458).

\providecommand{\bysame}{\leavevmode\hbox to3em{\hrulefill}\thinspace}
\providecommand{\href}[2]{#2}


\begin{thebibliography}{3}


\bibitem{BarbourPakes}
{\sc BARBOUR, A.~D. and PAKES, A.~G.} (1979). {Limit theorems for the simple
branching process allowing immigration. {II}. {T}he case of infinite
offspring mean}. {\em Adv. in Appl. Probab.} \textbf{11}, no.~1, 63--72.


\bibitem{Barczy}
{\sc BARCZY, M., BEN ALAYA, M., KEBAIER, A. and PAP, G.} (2018).
{Asymptotic properties of maximum likelihood estimator for the growth
rate for a jump-type cir process based on continuous time observations}. {\em
Stochastic Process. Appl.} \textbf{128}, no.~4, 1135--1164.



\bibitem{Ber96}
BERTOIN, J. (1996). {\em L\'evy processes}, Cambridge Tracts in Math, vol.
121. 

\bibitem{BlackwellDubins}
BLACKWELL, D. and DUBINS, L.~E (1983). {An extension of Skorohod’s almost sure representation theorem}.
	{\em Proc. Amer. Math. Soc.}  \textbf{89}, 691--692.

\bibitem{Bingham87}
{\sc BINGHAM, N.~H., GOLDIE, C.~M. and TEUGELS, J.~L.} (1987). {\it Regular variation},
Encyclopedia Math. Appl., Cambridgeshire, New York.

\bibitem{BINGHAM1976217}
BINGHAM, N.~H. (1976). {Continuous branching processes and spectral positivity}.
{\em Stochastic Process. Appl.} \textbf{4}, no.~3, 217--242.

\bibitem{zbMATH06817592}
{\sc CHAZAL, M., LOEFFEN, R. and PATIE, P.} (2018). Smoothness of continuous state branching with immigration semigroups. 
{\em J.Math.Anal.Appl.} \textbf{459}, no.~2, 619--660.

\bibitem{Cohn}
COHN, H. (1977). {On the convergence of the supercritical branching processes
with immigration}. {\em J. Appl. Probab.} \textbf{14}, no.~2,
387--390.

\bibitem{DawsonLi}
{\sc DAWSON, D.~A. and Li, Z.} (2012). {Stochastic equations, flows and
measure-valued processes}. {\em Ann. Probab. }\textbf{40}, no.~2, 813--857.


\bibitem{Duhalde}
{\sc DUHALDE, X., FOUCART, C. and MA, C.} (2014). {On the hitting times of
continuous-state branching processes with immigration}. {\em Stochastic Process.
Appl.} \textbf{124}, no.~12, 4182--4201. 

\bibitem{duffie2003}
{\sc DUFFIE, D., FILIPOVI\'{C}, D. and SCHACHERMAYER, M.} (2003). {Affine processes and
applications in finance}. {\em Ann. Appl. Probab.} \textbf{13}, no.~3,
984--1053.



\bibitem{DuqLab}
{\sc DUQUESNE, T and LABB\'E, C.} (2014). {{On the Eve property for CSBP}}. {\em Electron. J. Probab.} \textbf{19}, no. 6, 1--31. 

\bibitem{Durrett}
DURRETT, R. (2010). {\it Probability: Theory and examples}, 4th~ed. Cambridge
University Press, USA.

\bibitem{Fontbona}
{\sc FITTIPALDI, M.~C. and FONTBONA, J.} (2012). {On {SDE} associated with
continuous-state branching processes conditioned to never be extinct}.
{\em Electron. Commun. Probab.} \textbf{17}, no. 49, 1--13. 

\bibitem{foucart2019}
{\sc FOUCART, C. and MA, C.} (2019). {Continuous-state branching processes,
extremal processes and super-individuals}. {\em Ann. Inst. H. Poincar\'{e} Probab.
Statist.} \textbf{55}, no.~2, 1061--1086.

\bibitem{MR3263091}
{\sc FOUCART, C. and URIBE BRAVO, G.} (2014). {Local extinction in
continuous-state branching processes with immigration}. {\em Bernoulli.} \textbf{20}, no.~4, 1819--1844. 

\bibitem{Grey}
GREY, D.~R. (1974). {Asymptotic behaviour of continuous time, continuous
state-space branching processes}. {\em J. Appl. Probab.} \textbf{11},
669--677.

\bibitem{Grey2}
\bysame. (1977). {Almost-sure convergence in markov branching process with
infinite mean}. {\em J. Appl. Probab.} \textbf{14}, 702--716.

\bibitem{Heathcote}
HEATHCOTE, C.~R. (1965). {A branching process allowing immigration}. {\em J. Roy.
Statist. Soc. Ser. B.} \textbf{27}, 138--143. 

\bibitem{heyde1970}
HEYDE, C.~C. (1970). {Extension of a result of {S}eneta for the super-critical
galton-watson process}. {\em Ann. Math. Statist.} \textbf{41}, no.~2,
739--742.

\bibitem{Jiao2017}
{\sc JIAO, Y., MA, C. and SCOTTI, C.} (2017). {Alpha-cir model with branching processes
in sovereign interest rate modeling}. {\em Finance Stoch.} no.~21, 789--813.


\bibitem{jiao20}
JIAO, Y (ed). (2020). {\it From Probability to Finance: Lecture Notes of BICMR Summer School on Financial Mathematics}, Springer Nature.


\bibitem{KAW}
{\sc KAWAZU, K. and WATANABE, S.} (1971). {Branching processes with immigration
and related limit theorems}. {\em Teor. Verojatnost. i Primenen.} \textbf{16}, 34--51.

\bibitem{KELLERRESSEL20122329}
{\sc KELLER-RESSEL, M. and MIJATOVI\'{C}, A.} (2012). {On the limit
distributions of continuous-state branching processes with immigration}.
{\em Stochastic Process. Appl.} \textbf{122}, no.~6, 2329--2345.


\bibitem{KyprianouPardo}
{\sc KYPRIANOU, A.~E. and PARDO, J.~C.} (2008). {Continuous-state branching processes and
self-similarity}. {\em J. Appl. Probab.} \textbf{45}, no.~4, 1140--1160.



\bibitem{Kyprianoubook}
KYPRIANOU, A.~E. (2014). {\it Fluctuations of {L}\'evy processes with
applications}, 2nd~edn. Universitext, Springer, Heidelberg.

\bibitem{Lambert2}
LAMBERT, A. (2007). {Quasi-stationary distributions and the continuous-state
branching process conditioned to be never extinct}. {\em Electron. J. Probab.}
\textbf{12}, no.~14, 420--446. 

\bibitem{Li2000}
Li, Z. (2000). {Asymptotic behaviour of continuous time and state branching
processes}. {\em J. Aust. Math. Soc.}
\textbf{68}, 68--84.

\bibitem{Li}
\bysame. (2011). {\em Measure-valued branching markov processes}. Springer.





\bibitem{ChunhuaLi}
{\sc Li, Z. and MA, C.} (2008). {Catalytic discrete state branching models and related
limit theorems}. {\em J. Theor. Probab.} no.~21, 936--965.


\bibitem{Pakes}
PAKES, A.~G. (1979). {Limit theorems for the simple branching process
allowing immigration. {I}. {T}he case of finite offspring mean}. {\em Adv. in
Appl. Probab.} \textbf{11}, no.~1, 31--62. 


\bibitem{Pinsky}
PINSKY, M.~A. (1972) {Limit theorems for continuous state branching processes
with immigration}. {\em Bull. Amer. Math. Soc.} \textbf{78}, 242--244.


\bibitem{Res87}
RESNICK, S.~I. (2008). {\it Extreme values, regular variation and point
processes}, Springer series in operations research and financial engineering,
Springer, New York, vol. 4.

\bibitem{MR0270460}
SENETA, E. (1970). {On the supercritical {{G}}alton-{{W}}atson process with
immigration}. {\em Math. Biosci.} \textbf{7}, 9--14. 

\end{thebibliography}
\end{document}